\documentclass[article,12pt,3p]{elsarticle}
\usepackage{natbib}
\usepackage{amssymb}
\usepackage{amsthm,amsmath}
\usepackage{multirow}
\usepackage{array}
\usepackage{booktabs}
\newtheorem{Lemma}{\textbf{Lemma}}
\newtheorem{Theorem}{\textbf{Theorem}}
\usepackage{lineno}
\usepackage{color}
\usepackage{xcolor,colortbl}
\usepackage{amsmath,dsfont}
\usepackage{setspace}
\usepackage{graphicx}
\usepackage{float}
\usepackage{caption}
\usepackage[linkcolor=blue, urlcolor=blue, citecolor=blue,
colorlinks, bookmarks]{hyperref}
\usepackage{amsfonts}
\usepackage{subcaption}

\numberwithin{Definition}{section}

\usepackage{titlesec}
\titleformat*{\subsection}{\bfseries}
\titleformat*{\subsubsection}{\bfseries}

\numberwithin{equation}{section}

\newtheorem{remark}{\textbf{Remark}}
\usepackage[mathscr]{eucal}
\numberwithin{Lemma}{section} 
\newtheorem{Example}{Example}
\numberwithin{Example}{section}

\numberwithin{Corollary}{section}
\journal{}
\begin{document}
\begin{frontmatter}
		
\title{High-order approximation to generalized Caputo  derivatives and generalized fractional advection-diffusion equations }

\author[add1]{Sarita Kumari}
\ead{saritakumari.rs.mat19@itbhu.ac.in}
\author[add1]{Rajesh K. Pandey \corref{cor}}
\ead{rkpandey.mat@iitbhu.ac.in}
\cortext[cor]{Corresponding author}
\address[add1]{Department of Mathematical Sciences, Indian Institute of Technology (BHU) Varanasi,\\ Varanasi-221005, Uttar Pradesh, India.} 
\author[add2]{R. P. Agarwal}
\ead{ravi.agarwal@tamuk.edu}
\address[add2]{Department of Mathematics, Texas A \& M University-Kingsville, Kingsville, Texas, USA.}
		
\begin{abstract}
In this article, a high-order time-stepping scheme based on the cubic interpolation formula is considered to approximate the generalized Caputo fractional derivative (GCFD). Convergence order for this scheme is $(4-\alpha)$, where $\alpha ~(0<\alpha<1)$ is the order of the GCFD. The local truncation error is also provided. Then, we adopt the developed scheme to establish a difference scheme for the solution of generalized fractional advection-diffusion equation with Dirichlet boundary conditions. Furthermore, we discuss about the stability and convergence of the difference scheme. Numerical examples are presented to examine the theoretical claims. The convergence order of the difference scheme is analyzed numerically, which is $(4-\alpha)$ in time and second-order in space. 
\end{abstract}
\end{frontmatter}
	

\section{Introduction}\label{Sec-1}

 This work include, the numerical solution of following generalized fractional advection-diffusion equation, which is studied in \cite{yadav2019high}
 \begin{align}\label{1.1}
 \begin{cases}
 &^{C}_{0}\mathcal{D}_{t;[\zeta(t),\omega(t)]}^{\alpha}U(x,t)= D\frac{\partial^{2}U(x,t)}{\partial x^{2}}-A\frac{\partial U(x,t)}{\partial x}+g(x,t),~~ x \in \Omega,~~~ t \in  (0,T],\\[1.1ex]
 &U(x,0)=U_{0}(x), \qquad x \in   \Bar{\Omega}=\Omega \cup \partial\Omega,\\[1ex]
 &U(0,t)=\rho_{1}(t),~~U(a,t)=\rho_{2}(t),\qquad t \in (0,T],
 \end{cases}
 \end{align} 
 where $\Omega=(0,a)$  is a bounded domain with boundary $\partial \Omega$, and the notation $^{C}_{0}\mathcal{D}_{t,[\zeta(t),\omega(t)]}^{\alpha}$ denotes the GCFD (defined in \cite{agrawal2012some}, and related references therein) with respect to $t$ of order $\alpha$ is
 \begin{align}
 ^{C}_{0}\mathcal{D}_{t;[\zeta(t),\omega(t)]}^{\alpha}U(t)=\frac{[\omega(t)]^{-1}}{\Gamma(1-\alpha)}\int_{0}^{t}\frac{[\omega(s)U(s)]'}{[\zeta(t)-\zeta(s)]^{\alpha}}~ds,~~~ 0<\alpha<1,
 \end{align}
 where $\partial_{s} = \frac{\partial}{\partial s}$,
  parameters $D>0$ is the diffusivity, $A>0$ is the advection constant and $U$ is the solute concentration, $g$, $U_{0}$, $\rho_{1}$ and $\rho_{2}$ are continuous functions on their respective domains with $U_{0}(0)=\rho_{1}(0)$ and $U_{0}(a)=\rho_{2}(0)$. Here scale, weight are sufficiently regular functions and our model \eqref{1.1} reduces to the diffusion problem when $A=0$. Advection-diffusion equation is basically a transport problem that transport a passive scalar quantity in a fluid flow. Due to diffusion and advection, this model represents physical phenomenon of species concentration for mass transfer and temperature in heat transfer; for more details, we refer to \cite{benson2000application,zhou2003application, baleanu2017chaos}, and \cite{dan2005solving,verwer1996implicit,Dehghan2004,mohebbi2010high, owolabi2020high} for further history and significance of advection-diffusion equation in physics, chemistry and biology. 
 \par
 Mostly used fractional derivatives in the problem formulation are the Riemann-Liouville and the Caputo derivatives \cite{podlubny1998fractional}. In the year 2012, the generalizations of fractional integrals and derivatives were discussed by Agrawal \cite{agrawal2012some}. Two functions, scale~ $\zeta(t)$ and weight $\omega(t)$ in one parameter, appear in the definition of the generalized fractional derivative of a function $U(t)$. If $\omega(t)=1$ and $\zeta(t)=t$ then generalized fractional derivative reduces to the Riemann-Liouville (R-L) and the Caputo derivative, whereas if $\omega(t)=1$, $\zeta(t)=ln(t)$, and  $\omega(t)=t^{\sigma \eta}$, $\zeta(t)=t^{\sigma}$ then it will convert to Hadamard \cite{anatoly2001hadamard}, and modified Erd\'elyi-Kober fractional derivatives, respectively. In \cite{gaboury2013some} studied the generalized form of R-L and the Hadamard fractional integrals, which is a special case of the Erd\'elyi-Kober generalized fractional derivative, and some properties of this operator. Atangana and Baleanu \cite{atangana2016new} discussed a new fractional derivative with non-singular kernel and used this derivative in the formation of fractional heat transfer model. So, we obtain different types of fractional derivatives for different choices of weight and scale functions. In generalized derivative, scale function $\zeta(t)$ manages the considered time domain, it can stretch or contract accordingly to capture the phenomena accurately over desired time range. The weight function $\omega(t)$ allows the events to be estimated differently at different time.\par
 Over the last decade, many numerical methods were investigated to approximate the Caputo fractional derivative. For example, Mustapha \cite{mustapha2020l1} presented $L1$ approximation formula to solve fractional reaction-diffusion equation and second order error bound discussed on non-uniform time meshes. Alikhanov \cite{alikhanov2015new} constructed $L2-1_{\sigma}$ formula to approximate Caputo fractional time derivative and then used this derived scheme in solving time fractional diffusion equation with variable coefficients. Abu Arqub  \cite{abu2018robin} considered reproducing kernel algorithm
 for approximate solution of the nonlinear time-fractional PDE’s with initial and Robin boundary conditions. Li and Yan \cite{li2018error} discussed the idea of \cite{lv2016error} (i.e. $L_{2}$ approximation formula for time discretization), also derived a new time discretization method with accuracy order $\mathcal {O}(\tau^{3-\alpha})$ and finite element method for spatial discretization.  Cao et al. \cite {cao2015high} presented a high-order approximation formula based on the cubic interpolation to approximate the Caputo derivative for the time fractional advection-diffiusion equation. Xu and Agrawal  \cite{xu2013numerical} used the finite difference method (FDM) to approximate the GCFD for solving the generalized fractional Burgers equation. Kumar et al. \cite{kumar2021numerical} presented $L1$ and $L2$ methods to approximate the generalized time fractional derivative which are defined as follows, respectively
 \begin{align}
& ^{C}_{0}\mathcal{D}_{t;[\zeta(t),\omega(t)]}^{\alpha}U(t)|_{t=t_{n}}=\frac{[\omega(t)]^{-1}}{\Gamma(1-\alpha)}\sum_{l=1}^{n}\bigg(\frac{\omega_{l}U_{l}-\omega_{l-1}U_{l-1}}{\zeta_{l}-\zeta_{l-1}}\bigg)\int_{t_{l-1}}^{t_{l}}[\zeta(t)-\zeta(s)]^{-\alpha}~(\partial_{s}\zeta(s))~ds+r_{1}^{n},\\ 
& ^{C}_{0}\mathcal{D}_{t;[\zeta(t),\omega(t)]}^{\alpha}U(t)|_{t=t_{n}}=\frac{[\omega(t)]^{-1}}{\Gamma(1-\alpha)}\sum_{l=1}^{n}\int_{t_{l-1}}^{t_{l}}[\zeta(t)-\zeta(s)]^{-\alpha}~\partial_{s}(\Pi_{l}\omega(s)U(s))~ds+r_{2}^{n}, 
 \end{align}
 where,
 \begin{align} 
 \Pi_{l}'(\omega(t)U(t))= \zeta'(t)\bigg\{\bigg[&\frac{2\zeta(t)-\zeta_{l}-\zeta_{l-1}}{(\zeta_{l-2}-\zeta_{l-1})(\zeta_{l-2}-\zeta_{l-1})}\bigg]\omega_{l-2}U_{l-2}+\bigg[\frac{2\zeta(t)-\zeta_{l}-\zeta_{l-2}}{(\zeta_{l-1}-\zeta_{l-2})(\zeta_{l-1}-\zeta_{l})}\bigg]\nonumber \\ &\omega_{l-1}U_{l-1}+\bigg[\frac{2\zeta(t)-\zeta_{l-1}-\zeta_{l-2}}{(\zeta_{l}-\zeta_{l-2})(\zeta_{l}-\zeta_{l-1})}\bigg]\omega_{l}U_{l} \bigg\}, \nonumber
 \end{align}
 where, the domain $[0,T]$ was discretized into $n$ equal subintervals i.e. $0=t_{0}<t_{1}<...<t_{N}=T$ with step-size $\tau=\frac{T}{N}$, and errors $r_{1}^{n}=\mathcal {O}(\tau^{2-\alpha}),~ r_{2}^{n}=\mathcal {O}(\tau^{3-\alpha})$ were shown in \cite{kumar2021numerical}. In this work, we discuss the numerical scheme for GCFD with convergence rate $(4-\alpha)$; for this accuracy we have to assume that $U'(t_{0})=0$, $U''(t_{0})=0$, $U'''(t_{0})=0.$ This idea is discussed in \cite {cao2015high} for Caputo derivative approximation, but the error bound was discussed directly. 
 \par Due to the non-local property of the fractional derivatives, the numerical solution of the fractional partial differential equations is a very difficult task \cite{diethelm2011efficient}. Several authors have presented some precise and efficient numerical methods for fractional advection-diffusion equation. For examples, Zheng et al. \cite{zheng2010note} used finite element method (FEM) for space fractional advection-diffusion equation. Mardani et al. \cite{mardani2018meshless} discussed meshless moving least square method for solving the time-fractional advection-diffusion equation with variable coefficients. Cao et al. \cite{cao2015high} proposed the higher-order approximation of the Caputo derivatives and further applied it in solving the fractional advection-diffusion equation. They used the Lagrange interpolation method to discretize time derivative and second-order central difference for the spatial derivatives. Li and Cai \cite{li2017high} considered a three steps process for the Caputo fractional derivative approximation, first two steps include the shifted Lubich formula derivation for infinite interval then for finite interval, and after that it is generalized to the Caputo derivative. Yadav et al.  \cite{yadav2019high} discussed Taylor expansion for the approximation of the generalized time-fractional derivative to solve generalized fractional advection-diffusion equation. Tian et al. \cite{tian2014polynomial} presented polynomial spectral collocation method for space fractional advection-diffusion equation. In \cite{zhuang2009numerical}, authors developed explicit and implicit Euler approximations to solve variable-order fractional advection-diffusion equation on finite domain. Singh et al. \cite{singhapproximation} investigated the numerical approximation of Caputo-Prabhakar derivative and then used this approximation to solve fractional advection-diffusion equation. \par
 Up to now, there Xu and Agrawal \cite{xu2013numerical} considered the FDM for approximation of the GCFD for the generalized fractional Burgers equation. In \cite{kumar2021numerical}, authors presented numerical scheme for the generalized fractional telegraph equation in time. Cao et al. \cite{cao2018finite} worked on the generalized time-fractional Kdv equation. Xu et al. \cite{Xu2013} considered the solution of generalized fractional diffusion equation. Owolabi and Pindza \cite{owolabi2022numerical}  used generalized Caputo-type fractional derivative for the numerical simulation of nonlinear chaotic maps. Odibat and Baleanu  \cite{odibat2020numerical} presented an adaptive predictor corrector algorithm to solve initial value problems with generalized Caputo fractional derivative.  The main focus of this paper is to present much higher order numerical scheme to approximate GCFD as compared to \cite{xu2014numerical} and \cite{yadav2019high}, and also to establish the error analysis in both time and space-discretization.  To the best of our knowledge, no work has been done yet for a third-order error bound of cubic interpolation formula to approximate the GCFD. \par
 The main contributions of this work are as follows: \\[1ex]
 (1) we extend the approximation method of Cao et al. \cite{cao2015high}  for approximating the GCFD and obtain the convergence order $(4-\alpha)$. Further, we show that the obtained scheme reduces to the approximation scheme discussed by Cao et al. \cite{cao2015high} for choice of the scale and the weight functions as $\zeta(t)=t$ and $\omega(t)=1.$ \\[1ex]
 (2) we establish the full error analysis of the presented higher-order numerical scheme for the generalized time fractional derivative by using Lagrange interpolation formula. \\[1ex]
 (3) we introduce some numerical results for different choices of scale and weight functions for the high-order time discretization scheme  with convergence order $\mathcal {O}(\tau^{4-\alpha})$ for all $\alpha \in (0,1)$ which acheives higher accuracy than the  numerical methods developed in Gao et al. \cite{gao2014new} for $\zeta(t) =t$ and $\omega(t)=1$.   \par  
 \par The remaining sections of the paper is arranged as follows: In Section \ref{Sec-3}, we discuss the $(4-\alpha)$-th order scheme to approximate the GCFD of order $\alpha$ of the function $U.$ In Section \ref{Sec-4}, a higher order difference scheme to solve generalized fractional advection-diffusion equation is presented. Stability and convergence analysis are also discussed in this section.  We present three numerical examples which illustrate the error and convergence order of our established numerical scheme in Section \ref{Sec-5}. Finally, Section \ref{Sec-6}, concludes some remarks.
\section{Numerical scheme for the generalized Caputo fractional derivative}\label{Sec-3}
Motivated by the research carried out in {\cite{cao2015high, kumar2021numerical}}, this section is devoted to presenting a high-order approximation formula for the generalized Caputo-type fractional derivative using cubic interpolation polynomials.  \\
Suppose that $U(t) \in C^{4}[0,T]$, and grid points $0=t_{0}<t_{1}<...<t_{N}=T$ with step length $\tau=t_{n}-t_{n-1}$ for $1 \leq n \leq N$. For simplicity, we use $g(s)=\omega(s)U(s),$ $U(t_{l})=U_{l}$ , $\omega(t_{l})=\omega_{l}$, and $\zeta(t_{l})=\zeta_{l}$. The generalized Caputo fractional derivative of order  $\alpha$ of the function $U(t)$ at grid point $t_{n}$ is given by,
\begin{align}\label{2.1} ^{C}_{0}\mathcal{D}_{t;[\zeta(t),\omega(t)]}^{\alpha}U_{n}
=\frac{[\omega_{n}]^{-1}}{\Gamma(1-\alpha)}\sum_{l=1}^{n}\int_{t_{l-1}}^{t_{l}}\frac{g'(s)}{[\zeta(t_{n})-\zeta(s)]^{\alpha}}~ds. 
\end{align}
On the first interval $[0,t_{1}]$ of domain, we use continuous linear polynomial $\Pi_{1}g(t)$ to approximate the  function $g(t)(=\omega(t)U(t)).$ Let $g(t_{l})=g_{l}$ and the difference operator $\nabla_{\tau}g_{l}=g_{l}-g_{l-1}$ for $l \geq 1.$ Then, we have 
\begin{align}
(\Pi_{1}g(t))'=\frac{(g_{1}-g_{0})}{\tau} = \frac{(\nabla_{\tau}g_{1})}{\tau}. \nonumber
\end{align}
Thus, Eq. \eqref{2.1} yields,
\begin{align}\label{3.3}
\frac{[\omega_{n}]^{-1}}{\Gamma(1-\alpha)}~\int_{0}^{t_{1}}[\zeta(t_{n})-\zeta(s)]^{-\alpha}~g'(s)~ds&= \frac{[\omega_{n}]^{-1}}{\Gamma(1-\alpha)}~\frac{(\nabla_{\tau}g_{1})}{\tau}\int_{0}^{t_{1}}[\zeta(t_{n})-\zeta(s)]^{-\alpha}ds+E_{\tau}^{1}&  \nonumber \\
&=a_{n-1}(\nabla_{\tau}g_{1})+E_{\tau}^{1},
\end{align}
where $E_{\tau}^{1}$ is the truncation error on first interval and coefficients for this approximation are
\begin{align}
a_{n-l}=\frac{[\omega_{n}]^{-1}}{\Gamma(2-\alpha)} \Bigg\{\frac{[\zeta(t_{n})-\zeta(t_{l-1})]^{1-\alpha}-[\zeta(t_{n})-\zeta(t_{l})]^{1-\alpha}}{[\zeta(t_{l})-\zeta(t_{l-1})]}\Bigg\}.  \nonumber
\end{align}
Here, we denote notation $\alpha_{0}^{(n)}=\frac{[\omega_{n}]^{-1}}{\Gamma(2-\alpha)},~ 1 \leq n \leq N.$ Then 
\begin{align}\label{3.4}
a_{n-l}=\alpha_{0}^{(n)}\Bigg\{\frac{[\zeta(t_{n})-\zeta(t_{l-1})]^{1-\alpha}-[\zeta(t_{n})-\zeta(t_{l})]^{1-\alpha}}{[\zeta(t_{l})-\zeta(t_{l-1})]}\Bigg\},~~~  1 \leq l \leq n. 
\end{align}
\begin{remark}
If the scale function $\zeta$ is a positive strictly increasing functions on the domain $[0,T]$, then following inequality holds
\begin{align}\label{3.5}
0<[\zeta(t_{n})-\zeta(t_{l})]^{1-\alpha}-[\zeta(t_{n})-\zeta(t_{l-1})]^{1-\alpha},~~~~1 \leq l \leq n.
\end{align}
Since, $t_{l} > t_{l-1},$ then it implies 
$\zeta(t_{l}) > \zeta(t_{l-1})$ for $1 \leq l \leq n,$ also $(1-\alpha)>0.$
\end{remark}
\begin{remark}
To estimate $a_{n-1}$, 
suppose that
$$ \Theta =[\zeta(t_{n})-\zeta(s)] \\
\Rightarrow  d\Theta  = - \zeta'(s) ~ ds \\
= -\bigg(\frac{\zeta(t_{l})-\zeta(t_{l-1})}{\tau}\bigg)~ds, ~~ s \in (t_{l-1},t_{l}), $$
therefore,
$$\int[\zeta(t_{n})-\zeta(s)]^{-\alpha}~ds=\frac{-\tau}{\zeta(t_{l})-\zeta(t_{l-1})}\frac{[\zeta(t_{n})-\zeta(s)]^{1-\alpha}}{(1-\alpha)}. $$
\end{remark} 
On the second interval $[t_{1},t_{2}]$, we use continuous quadratic polynomial $\Pi_{2}g(t)$ to approximate the function $g(t)$, then we get
\begin{align}
(\Pi_{2}g(t))'=\frac{(2t-t_{0}-t_{1})}{2\tau^{2}}g_{2}-\frac{(2t-t_{0}-t_{2})}{\tau^{2}}g_{1}+\frac{(2t-t_{1}-t_{2})}{2\tau^{2}}g_{0}. \nonumber
\end{align}
Thus from Eq. \eqref{2.1}, we get,
\begin{align}\label{2.5}
\frac{[\omega_{n}]^{-1}}{\Gamma(1-\alpha)}~\int_{t_{1}}^{t_{2}}[\zeta(t_{n})-\zeta(s)]^{-\alpha}~g'(s)~ds  &=\frac{[\omega_{n}]^{-1}}{\Gamma(1-\alpha)}~\int_{t_{1}}^{t_{2}}[\zeta(t_{n})-\zeta(s)]^{-\alpha}~(\Pi_{2}g(s))'~ds+E_{\tau}^{2}  \nonumber \\
&=a_{n-2}(g_{2}-g_{1})+b_{n-2}(g_{2}-2g_{1}+g_{0})+E_{\tau}^{2} \nonumber\\[1ex] 
&=a_{n-2}\big(\nabla_{\tau}g_{2}\big)+b_{n-2}\big(\nabla_{\tau}^{2}g_{2}\big)+E_{\tau}^{2}. 
\end{align}
Here, truncation error on second interval is $E_{\tau}^{2},$  and 
\begin{align}\label{3.8}
b_{n-l}=\alpha_{0}^{(n)}\Bigg\{&\frac{1}{(2-\alpha)}\bigg[\frac{[\zeta(t_{n})-\zeta(t_{l-1})]^{2-\alpha}-[\zeta(t_{n})-\zeta(t_{l})]^{2-\alpha}}{[\zeta(t_{l})-\zeta(t_{l-1})]^{2}}\bigg]\nonumber \\ &-\frac{1}{2}\bigg[\frac{[\zeta(t_{n})-\zeta(t_{l-1})]^{1-\alpha}+[\zeta(t_{n})-\zeta(t_{l})]^{1-\alpha}}{[\zeta(t_{l})-\zeta(t_{l-1})]}\bigg]\Bigg \},~~~~ 2 \leq l \leq n.
\end{align}
On the other subdomains $(l \geq 3)$, we use cubic interpolation polynomial $\Pi_{l}g(t)$ to approximate the function $g(t)$ using four points $(t_{l-3},g_{l-3}), (t_{l-2},g_{l-2}), (t_{l-1},g_{l-1}), (t_{l},g_{l})$.\\[0.5ex]
As we know cubic interpolation polynomial is defined as,
\begin{align}
\Pi_{l}g(t)=\sum_{r=0}^{3}g_{l-r}\prod_{s=0,{s \neq r}}^{3}\bigg(\frac{t-t_{l-s}}{t_{l-r}-t_{l-s}}\bigg), \nonumber
\end{align}
then we get,
\begin{align}
(\Pi_{l}g(t))'&=g_{l-3}\frac{(t-t_{l-2})(t_{l}+t_{l-1}-2t)+(t-t_{l-1})(t_{l}-t)}{6 \tau^{3}} \nonumber \\
&+g_{l-2}\frac{(t-t_{l-3})(2t-t_{l-1}-t_{l})+(t-t_{l})(t-t_{l-1})}{2\tau^{3}} \nonumber \\
&+g_{l-1}\frac{(t-t_{l-2})(t_{l-3}+t_{l}-2t)+(t-t_{l-3})(t_{l}-t)}{2 \tau^{3}} \nonumber \\
&+g_{l}\frac{(t-t_{l-2})(2t-t_{l-3}-t_{l-1})+(t-t_{l-1})(t-t_{l-3})}{6\tau^{3}}. \nonumber
\end{align}
Therefore, from Eq. \eqref{2.1}, we have,
\begin{align}\label{2.7}
&\frac{[\omega_{n}]^{-1}}{\Gamma(1-\alpha)}\sum_{l=3}^{n}\int_{t_{l-1}}^{t_{l}}[\zeta(t_{n})-\zeta(s)]^{-\alpha}g'(s)~ds  \nonumber \\ \nonumber  &= \frac{[\omega_{n}]^{-1}}{\Gamma(1-\alpha)}\sum_{l=3}^{n}\int_{t_{l-1}}^{t_{l}}[\zeta(t_{n})-\zeta(s)]^{-\alpha}(\Pi_{l}g(s))'~ds+E_{\tau}^{n}  \nonumber \\
&=\sum_{l=3}^{n}\bigg[ A_{1,n-l}~g_{l}+A_{2,n-l}~g_{l-1}  +A_{3,n-l}~g_{l-2}+A_{4,n-l}~g_{l-3}\bigg]+E_{\tau}^{n}, 
\end{align}
where $E_{\tau}^{n}~(3 \leq n \leq N)$ is the truncation error, and
\begin{align}
A_{1,n-l}~=~&\alpha_{0}^{(n)}\Bigg\{\frac{1}{6}\bigg[\frac{2[\zeta(t_{n})-\zeta(t_{l-1})]^{1-\alpha}-11[\zeta(t_{n})-\zeta(t_{l})]^{1-\alpha}}{[\zeta(t_{l})-\zeta(t_{l-1})]}\bigg] \nonumber \\
&+\frac{1}{(2-\alpha)}\bigg[\frac{[\zeta(t_{n})-\zeta(t_{l-1})]^{2-\alpha}-2[\zeta(t_{n})-\zeta(t_{l})]^{2-\alpha}}{[\zeta(t_{l})-\zeta(t_{l-1})]^{2}}\bigg] \nonumber \\
&+\frac{1}{(2-\alpha)(3-\alpha)}\bigg[\frac{[\zeta(t_{n})-\zeta(t_{l-1})]^{3-\alpha}-[\zeta(t_{n})-\zeta(t_{l})]^{3-\alpha}}{[\zeta(t_{l})-\zeta(t_{l-1})]^{3}}\bigg]\Bigg \}, 
\end{align}
\begin{align}
A_{2,n-l}~=~&\alpha_{0}^{(n)}\Bigg\{\frac{1}{2}\bigg[\frac{6[\zeta(t_{n})-\zeta(t_{l})]^{1-\alpha}+[\zeta(t_{n})-\zeta(t_{l-1})]^{1-\alpha}}{[\zeta(t_{l})-\zeta(t_{l-1})]}\bigg] \nonumber \\
&+\frac{1}{(2-\alpha)}\bigg[\frac{5[\zeta(t_{n})-\zeta(t_{l})]^{2-\alpha}-2[\zeta(t_{n})-\zeta(t_{l-1})]^{2-\alpha}}{[\zeta(t_{l})-\zeta(t_{l-1})]^{2}}\bigg] \nonumber \\
&+\frac{3}{(2-\alpha)(3-\alpha)}\bigg[\frac{[\zeta(t_{n})-\zeta(t_{l})]^{3-\alpha}-[\zeta(t_{n})-\zeta(t_{l-1})]^{3-\alpha}}{[\zeta(t_{l})-\zeta(t_{l-1})]^{3}}\bigg]\Bigg\}, 
\end{align}
\begin{align}
A_{3,n-l}~=~&\alpha_{0}^{(n)}\Bigg\{-\frac{1}{2}\bigg[\frac{2[\zeta(t_{n})-\zeta(t_{l-1})]^{1-\alpha}+3[\zeta(t_{n})-\zeta(t_{l})]^{1-\alpha}}{[\zeta(t_{l})-\zeta(t_{l-1})]}\bigg] \nonumber \\
&+\frac{1}{(2-\alpha)}\bigg[\frac{[\zeta(t_{n})-\zeta(t_{l-1})]^{2-\alpha}-4[\zeta(t_{n})-\zeta(t_{l})]^{2-\alpha}}{[\zeta(t_{l})-\zeta(t_{l-1})]^{2}}\bigg] \nonumber \\
&+\frac{3}{(2-\alpha)(3-\alpha)}\bigg[\frac{[\zeta(t_{n})-\zeta(t_{l-1})]^{3-\alpha}-[\zeta(t_{n})-\zeta(t_{l})]^{3-\alpha}}{[\zeta(t_{l})-\zeta(t_{l-1})]^{3}}\bigg]\Bigg\},  
\end{align}
\begin{align}
A_{4,n-l}~=~&\alpha_{0}^{(n)}\Bigg\{\frac{1}{6}\bigg[\frac{[\zeta(t_{n})-\zeta(t_{l-1})]^{1-\alpha}+2[\zeta(t_{n})-\zeta(t_{l})]^{1-\alpha}}{[\zeta(t_{l})-\zeta(t_{l-1})]}\bigg] \nonumber \\
&+\frac{1}{(2-\alpha)}\bigg[\frac{[\zeta(t_{n})-\zeta(t_{l})]^{2-\alpha}}{[\zeta(t_{l})-\zeta(t_{l-1})]^{2}}\bigg] \nonumber \\
&+\frac{1}{(2-\alpha)(3-\alpha)}\bigg[\frac{[\zeta(t_{n})-\zeta(t_{l})]^{3-\alpha}-[\zeta(t_{n})-\zeta(t_{l-1})]^{3-\alpha}}{[\zeta(t_{l})-\zeta(t_{l-1})]^{3}}\bigg]\Bigg\}, 
\end{align}
where, $3 \leq l \leq n $.
After simplifying the Eq. \eqref{2.7}, we obtain the following form 
\begin{align}\label{2.12}
&\frac{[\omega_{n}]^{-1}}{\Gamma(1-\alpha)}\sum_{l=3}^{n}\int_{t_{l-1}}^{t_{l}}[\zeta(t_{n})-\zeta(s)]^{-\alpha}~g'(s)~ds \nonumber \\ &=\sum_{l=3}^{n}\big[ a_{n-l}(g_{l}-g_{l-1})+b_{n-l}(g_{l}-2g_{l-1} +g_{l-2})\nonumber \\[1ex] &+c_{n-l}(g_{l}-3g_{l-1}+3g_{l-2}-g_{l-3})\big]+E_{\tau}^{n} \nonumber \\[1.5ex] 
&=\sum_{l=3}^{n}\big[a_{n-l}(\nabla_{\tau}g_{l})+b_{n-l}(\nabla_{\tau}^{2}g_{l})+c_{n-l}(\nabla_{\tau}^{3}g_{l})\big]+E_{\tau}^{n}.
\end{align}
Here, we introduce another coefficient $c_{n-l},$ which is defined as 
\begin{align}
c_{n-l}=\alpha_{0}^{(n)}\Bigg\{&\frac{1}{(2-\alpha)(3-\alpha)}\bigg[\frac{[\zeta(t_{n})-\zeta(t_{l-1})]^{3-\alpha}-[\zeta(t_{n})-\zeta(t_{l})]^{3-\alpha}}{[\zeta(t_{l})-\zeta(t_{l-1})]^{3}}\bigg] \nonumber \\ \nonumber&-\frac{1}{(2-\alpha)}\bigg[\frac{[\zeta(t_{n})-\zeta(t_{l})]^{2-\alpha}}{[\zeta(t_{l})-\zeta(t_{l-1})]^{2}}\bigg]\\  &-\frac{1}{6}\bigg[\frac{[\zeta(t_{n})-\zeta(t_{l-1})]^{1-\alpha}+2[\zeta(t_{n})-\zeta(t_{l})]^{1-\alpha}}{[\zeta(t_{l})-\zeta(t_{l-1})]}\bigg]\Bigg\},~~~ 3 \leq l \leq n,
\end{align}  
and coefficients $a_{n-l},$ $b_{n-l}$ are defined in Eqs. \eqref{3.4}, \eqref{3.8}, respectively and Eq. \eqref{2.12} gives more compact form of Eq. \eqref{2.7}. Such forms of coefficients were missing in \cite{cao2015high}. From this we can easily discuss properties of coefficients.     \\[1ex]
Motivated by \cite{cao2015high} (developed for Caputo derivative), a new numerical scheme for the generalized Caputo-type fractional derivative of order $\alpha$ of the function $U(t)$ at grid point $t_{n}$ with the help of the equations \eqref{3.3}, \eqref{2.5}, \eqref{2.7}, is defined by
\begin{align}\label{ab}
^\mathcal{H}\mathcal{D}_{t;[\zeta(t),\omega(t)]}^{\alpha}U(t)\big|_{t=t_{n}}&=\frac{[\omega_{n}]^{-1}}{\Gamma(1-\alpha)}\int_{0}^{t_{n}}\frac{g'(s)}{[\zeta(t_{n})-\zeta(s)]^{\alpha}}~ds \nonumber \\[1ex]
&=\frac{[\omega_{n}]^{-1}}{\Gamma(1-\alpha)}\Bigg[\int_{0}^{t_{1}}\frac{g'(s)}{[\zeta(t_{n})-\zeta(s)]^{\alpha}}~ds+\int_{t_{1}}^{t_{2}}\frac{g'(s)}{[\zeta(t_{n})-\zeta(s)]^{\alpha}}~ds\nonumber \\[1ex] &+\sum_{l=3}^{n}\int_{l-1}^{t_{l}}\frac{g'(s)}{[\zeta(t_{n})-\zeta(s)]^{\alpha}}~ds\Bigg] \nonumber \\[1ex]
&= \sum_{l=0}^{n} \lambda_{l}~\omega_{n-l}~U_{n-l},
\end{align}
with $g(s)=\omega(s)U(s)$.
\begin{Lemma}
	For any $\alpha \in (0,1)$ and $U(t) \in \mathcal{C}^{4}[0,T], $ then 
	\begin{align}
	^{C}_{0}\mathcal{D}_{t;[\zeta(t),\omega(t)]}^{\alpha}U_{n}=^\mathcal{H}\mathcal{D}_{t;[\zeta(t),\omega(t)]}^{\alpha}U_{n}+E_{\tau}^{n},~ n=1,2 ,...,N,\nonumber 
	\end{align}
	where, $^\mathcal{H}\mathcal{D}_{t;[\zeta(t),\omega(t)]}^{\alpha}$ is the approximation of GCFD and  $|E_{\tau}^{n}|=\mathcal{O}(\tau^{4-\alpha}).$
\end{Lemma}
For distinct value of $n$, the coefficients in \eqref{ab} can be expressed as below,\\
For $n=1$, 
\begin{center}
	$\begin{cases}
	\lambda_{0}=a_{0},\\
	\lambda_{1}=-a_{0}. 
	\end{cases}$
\end{center} 
For $n=2$, 
\begin{center}
	$\begin{cases}
	\lambda_{0}=a_{0}+b_{0},\\
	\lambda_{1}=a_{1}-a_{0}-2b_{0},\\
	\lambda_{2}=-a_{1}+b_{0}.
	\end{cases}$
\end{center}
For $n=3$, 
\begin{center}
	$\begin{cases}
	\lambda_{0}=A_{1,0},\\
	\lambda_{1}=A_{2,0}+a_{1}+b_{1},\\
	\lambda_{2}=A_{3,0}+a_{2}-a_{1}-2b_{1},\\
	\lambda_{3}=A_{4,0}-a_{2}+b_{1}.
	\end{cases}$
\end{center}
For $n=4$,
\begin{center}
	$\begin{cases}
	\lambda_{0}=A_{1,0},\\
	\lambda_{1}=A_{1,1}+A_{2,0},\\
	\lambda_{2}=A_{2,1}+A_{3,0}+a_{2}+b_{2},\\
	\lambda_{3}=A_{3,1}+A_{4,0}+a_{3}-a_{2}-2b_{2},\\
	\lambda_{4}=A_{4,1}-a_{3}+b_{2}.
	\end{cases}$
\end{center} 
For $n=5$,
\begin{center}
	$\begin{cases}
	\lambda_{0}=A_{1,0},\\
	\lambda_{1}=A_{1,1}+A_{2,0},\\
	\lambda_{2}=A_{1,2}+A_{2,1}+A_{3,0},\\
	\lambda_{3}=A_{2,2}+A_{3,1}+A_{4,0}+a_{3}+b_{3},\\
	\lambda_{4}=A_{3,2}+A_{4,1}+a_{4}-a_{3}-2b_{3},\\
	\lambda_{5}=A_{4,2}-a_{4}+b_{3}.
	\end{cases}$
\end{center}
 For $n \geq 6$,
\begin{align}\label{3.17}
\begin{cases}
\lambda_{0}=A_{1,0},\\
\lambda_{1}=A_{1,1}+A_{2,0},\\
\lambda_{2}=A_{1,2}+A_{2,1}+A_{3,0},\\
\lambda_{l}=A_{1,l}+A_{2,l-1}+A_{3,l-2}+A_{4,l-3}~~~(3 \leq l \leq n-3), \\
\lambda_{n-2}=a_{n-2}+b_{n-2}+A_{2,n-3}+A_{3,n-4}+A_{4,n-5},\\
\lambda_{n-1}=a_{n-1}-a_{n-2}-2b_{n-2}+A_{3,n-3}+A_{4,n-4},\\
\lambda_{n}=-a_{n-1}+b_{n-2}+A_{4,n-3}.
\end{cases}
\end{align}
\begin{Lemma}
	If the scale function $\zeta$ fulfills \eqref{3.5}, and the weight function $\omega$ is non-negative and non-decreasing on uniform time grids, then\\
	(i)\cite{cao2018finite} Linear approximation coefficient satisfies, $1 \leq l \leq n,$
	\begin{align}
	0<a_{n-1}<...<a_{n-l}<a_{n-l-1}<...<a_{1}<a_{0}. \nonumber 
	\end{align}  	
	(ii) Quadratic approximation coefficient satisfies, $2 \leq l \leq n,$ 
	\begin{align}
	0<b_{n-2}<...<b_{n-l}<b_{n-l-1}<...<b_{1}<b_{0}. \nonumber
	\end{align}
\end{Lemma}
\begin{proof}(i) If scale function $\zeta$ is continuous on respective domain then by using mean-value theorem, there exist $\hat{x}_{l} \in (t_{l-1},t_{l})$ such that
	\begin{align}
	\frac{1}{\tau}\int_{t_{l-1}}^{t_{l}}[\zeta(t_{n})-\zeta(s)]^{-\alpha}ds = [\zeta(t_{n})-\zeta(\hat{x}_{l})]^{-\alpha},~~~~ 1 \leq l \leq n.
	\end{align}
	Since, $[\zeta(t_{n})-\zeta(s)]^{-\alpha}$ is a monotone increasing function, we easily get our required result.\\[1.5ex]
	(ii) Let $\eta(s)=[\zeta(t_{n})-\zeta(s)]^{1-\alpha}$, suppose scale function $\zeta$ is sufficiently smooth on domain $[0,T]$ then mean-value theorem yields 
	\begin{align}\label{2.20}
	2 \int_{t_{l-1}}^{t_{l}}\eta(s)ds-\big(\eta(t_{l-1})+\eta(t_{l})\big)&=2\eta(\tilde{x}_{l})-\big(\eta(t_{l-1})+\eta(t_{l})\big), ~~~\tilde{x_{l}} \in (t_{l-1},t_{l}), \nonumber \\
	&= -\theta \big(\eta'(t_{l})-\eta'(t_{l-1})\big), \qquad \qquad 0< \theta <1 \nonumber \\ \nonumber
	&= -\theta \tau \eta''(\nu_{l}), ~~~~~~~~~~~~~~~~~~~~~~~~ \nu_{l} \in (t_{l-1},t_{l}),
	\nonumber \\
	&= \frac{\theta \alpha (1-\alpha)}{\tau}[\zeta(t_{l})-\zeta(t_{l-1})]^{2}[\zeta(t_{n})-\zeta(\nu_{l})]^{-\alpha-1}>0. 
	\end{align}
	Using \eqref{2.20}, we can easily get that $b_{n-l}>0$ for positive strictly increasing weight function $\omega(t)$.
	Since, $[\zeta(t_{n})-\zeta(s)]^{-\alpha-1}$ is a monotone increasing function on temporal domain $[0,T]$, so we get the desired result.
\end{proof}
\begin{Lemma}
	Suppose that the scale function $\zeta$ is positive and strictly increasing, and the weight function $\omega$ is non-negative and non-decreasing, then for each $\alpha \in (0,1)$, the following conditions hold for coefficients in \eqref{3.17}\\
	(1) $\lambda_{0}>0$,~~ $\forall~ n \geq 1$,\\
	(2) $\sum_{l=0}^{n}\lambda_{l}=0$.
\end{Lemma}
\begin{proof} (1) If $n=1$, then
	\begin{align} &\lambda_{0}=a_{0}=\alpha_{0}^{(n)}[\zeta_{1}-\zeta_{0}]^{-\alpha} ~~, \nonumber
	\end{align}
	as scale function is strictly increasing, therefore $\zeta_{n-1}<\zeta_{n}$, for $n \geq 1$. Which implies that  $\lambda_{0}=a_{0}>0.$\\[1ex]
	If $n=2$, then ~
	$
	\lambda_{0}=a_{0}+b_{0}, \nonumber 
	$ \\
	since, 
	\begin{align}
	b_{0}=\alpha_{0}^{(n)}\bigg\{\bigg(\frac{1}{2-\alpha}-\frac{1}{2}\bigg)[\zeta_{2}-\zeta_{1}]^{-\alpha}\bigg\}>0, \nonumber
	\end{align}
	and we have already shown that $a_{0}>0$, therefore $\lambda_{0}=a_{0}+b_{0}>0.$ \\[0.5ex]
	If $n \geq 3$, then for $0<\alpha<1$, 
	\begin{align}
	\lambda_{0}=A_{1,0}=\alpha_{0}^{(n)}\bigg\{\bigg(\frac{1}{3}+\frac{1}{(2-\alpha)}+\frac{1}{(2-\alpha)(3-\alpha)}\bigg)[\zeta_{n}-\zeta_{n-1}]^{-\alpha}\bigg\}>0. \nonumber
	\end{align}
	\\
	(2) If $n=1$, then $\lambda_{0}=-\lambda_{1}=a_{0}>0,$ for $0<\alpha<1.$ It implies that $\lambda_{0}+\lambda_{1}=0.$\\\\
	If $n=2$, then there exist a $\alpha \in (0,1)$, by numerical analysis  $$\lambda_{0}+\lambda_{1}+\lambda_{2}=(a_{0}+b_{0})+(a_{1}-a_{0}+2b_{0})+(-a_{1}+b_{0})=0.$$ 
	If $n\geq 3$, then 
	\begin{align}
	\sum_{l=0}^{n}\lambda_{l}&=A_{1,0}+A_{1,1}+A_{2,0}+A_{1,2}+A_{2,1}+A_{3,0} \nonumber\\ \nonumber&+\sum_{l=3}^{n-3}(A_{1,l}+A_{2,l-1}+A_{3,l-2}+A_{4,l-3})+a_{n-2}+b_{n-2}\nonumber\\&+A_{2,n-3}+A_{3,n-4}+A_{4,n-5}+a_{n-1}-a_{n-2}-2b_{n-2}+A_{3,n-3}\nonumber\\&+A_{4,n-4}-a_{n-1}+b_{n-2}+A_{4,n-3} \nonumber\\
	&=\sum_{j=0}^{n-3}(A_{1,j}+A_{2,j}+A_{3,j}+A_{4,j})=0. \nonumber
	\end{align}
\end{proof}
\begin{Lemma}\label{L1}
	If scale function $\zeta$ is a Lipschitz function on interval~ $[t_{l-1},t_{l}]$ with Lipschitz constant L, then
	\begin{align}
	|\zeta_{l}-\zeta_{l-1}| \leq L\tau, ~~~1 \leq l \leq n. \nonumber
	\end{align}
\end{Lemma}
\subsection{Truncation error for generalized Caputo derivative term}
For truncation error of approximation of the generalized Caputo derivative defined in \eqref{ab}, for simplicity, suppose that function $g(t)=\omega(t)U(t)$ such that $g(t) \in C^{4}((0,T])$, and $\zeta(t)=v$ this implies $t=\zeta^{-1}(v).$ Therefore,   $g(v)=\omega(\zeta^{-1}(v))U(\zeta^{-1}(v))$.
\begin{Theorem}
	A triangle inequality gives the bound 
	\begin{align}
	|^{C}_{0}\mathcal{D}_{t;[\zeta(t),\omega(t)]}^{\alpha}g_{n}-^\mathcal{H}\mathcal{D}_{t;[\zeta(t),\omega(t)]}^{\alpha}g_{n}| \leq \sum_{n=1}^{N}|E_{\tau}^{n}|,\nonumber 
	\end{align}
	with 
	\begin{align}
	E_{\tau}^{n} = \frac{[\omega_{n}]^{-1}}{\Gamma(1-\alpha)}\int_{\zeta_{l-1}}^{\zeta_{l}}[\zeta_{n}-v]^{-\alpha}[g(v)-\Pi_{l}g(v)]'~dv, ~~~1 \leq l \leq n. \nonumber
	\end{align}
	\begin{align}
	(1)~ |E_{\tau}^{1}|&\leq \frac{\alpha[\omega_{n}]^{-1}}{\Gamma(1-\alpha)}\bigg[\frac{1}{8\alpha}+\frac{1}{2(1-\alpha)(2-\alpha)}\bigg]\max_{t_{0} \leq v \leq t_{1}}|g^{(2)}(v)|L^{2-\alpha}(\tau)^{2-\alpha},~~ n=1, \nonumber \\ \nonumber \\
	(2)~ |E_{\tau}^{2}|&\leq \frac{\alpha[\omega_{n}]^{-1}}{\Gamma(1-\alpha)}\bigg\{\frac{1}{12}~\max_{t_{0} \leq v \leq t_{1}}|g^{(2)}(v)|(t_{2}-t_{1})^{-\alpha-1}L^{2-\alpha}(\tau)^{3}+\bigg[\frac{1}{12} \nonumber \\&+\frac{1}{3(1-\alpha)(2-\alpha)}\bigg(\frac{1}{2}+\frac{1}{(3-\alpha)}\bigg)\bigg]\max_{t_{0} \leq  v \leq t_{2}}|g^{(3)}(v)|L^{3-\alpha}(\tau)^{3-\alpha}\bigg\},~~n=2, \nonumber \\ \nonumber \\ \nonumber
	(3)~ |E_{\tau}^{n}|&\leq \frac{\alpha[\omega_{n}]^{-1}}{\Gamma(1-\alpha)}\bigg\{\frac{1}{72}~\max_{t_{0} \leq x \leq t_{2}}|g^{(3)}(x)|(t_{n}-t_{2})^{-\alpha-1}L^{3-\alpha}(\tau)^{4}+\bigg[\frac{3}{128\alpha}\\&+\frac{1}{12(1-\alpha)(2-\alpha)}\bigg(1+\frac{3}{(3-\alpha)}+\frac{3}{(3-\alpha)(4-\alpha)}\bigg)\bigg]\max_{t_{0} \leq  x_{1} \leq t_{n}}|g^{(4)}(x_{1})|L^{4-\alpha}(\tau)^{4-\alpha}  \bigg\}, \nonumber \\ \nonumber
	& \qquad\qquad\qquad\qquad\qquad\qquad\qquad\qquad\qquad\qquad\qquad\qquad\qquad\qquad\qquad\qquad~n\geq 3. \nonumber
	\end{align}
\end{Theorem}
\begin{proof}
	$(1)$\cite{kumar2021numerical}~ For ~$n=1,$ the scheme \eqref{ab} is linear approximation of GCFD, and convergence order for this approximation formula is $\mathcal{O}(\tau^{2-\alpha}).$ \\[1.5ex] 
	$(2)$\cite{kumar2021numerical}~ For ~$n=2,$ the scheme \eqref{ab} is quadratic approximation of GCFD, and here convergence order of approximation formula is $\mathcal{O}(\tau^{3-\alpha}).$\\
	$(3)$~ For ~$n \geq 3,$ By Lagrange interpolation remainder theorem, we use quadratic interpolation function $\Pi_{2}g(v)$ to interpolate $g(v)$ using node points $(\zeta_{0},g_{0})$, $(\zeta_{1},g_{1})$, $(\zeta_{2},g_{2})$ on interval $[\zeta_{0},\zeta_{2}]$  and cubic interpolation function $\Pi_{l}g(v)$ depends on $(\zeta_{l-3},g_{l-3})$, $(\zeta_{l-2},g_{l-2})$, $(\zeta_{l-1},g_{l-1})$, $(\zeta_{l},g_{l})$ to interpolate $ g(v)$ on~ $ [\zeta_{l-3},\zeta_{l}]$ as follows 
	\begin{align}\label{2.21}
	&g(v)-\Pi_{{2}}g(v)=\frac{g^{(3)}(\eta_{1})}{3!}~(v-\zeta_{0})(v-\zeta_{1})(v-\zeta_{2}),~v \in [\zeta_{0},\zeta_{2}],~\eta_{1} \in (\zeta_{0},\zeta_{2}),\\
	&g(v)-\Pi_{l}g(v)=\frac{g^{(4)}(\eta_{l})}{4!}~(v-\zeta_{l-3})(v-\zeta_{l-2})(v-\zeta_{l-1})(v-\zeta_{l}),~ v \in [\zeta_{l-3},\zeta_{l}],~\eta_{l} \in (\zeta_{l-3}, \zeta_{l}),
	\end{align}
	where $3 \leq l \leq n.$ \\
	Now,
	\begin{align}\label{2.23}
	E_{\tau}^{n}~=~\frac{[\omega_{n}]^{-1}}{\Gamma(1-\alpha)}&\bigg[\int_{\zeta_{0}}^{\zeta_{2}}[g(v)-\Pi_{2}g(v)]'~[\zeta_{n}-v]^{-\alpha}~dv \nonumber \\&+\sum_{l=3}^{n}\int_{\zeta_{l-1}}^{\zeta_{l}}[g(v)-\Pi_{l}g(v)]'~[\zeta_{n}-v]^{-\alpha}~dv\bigg] \nonumber \\
	=~\frac{[\omega_{n}]^{-1}}{\Gamma(1-\alpha)}&\bigg\{[g(v)-\Pi_{2}g(v)]~[\zeta_{n}-v]^{-\alpha}\big|_{\zeta_{0}}^{\zeta_{2}} \nonumber \\ &-\alpha\int_{\zeta_{0}}^{\zeta_{2}}[g(v)-\Pi_{2}g(v)]~[\zeta_{n}-v]^{-\alpha-1}~dv \nonumber \\
	&+\sum_{l=3}^{n}\bigg[g(v)-\Pi_{l}g(v)]~[\zeta_{n}-v]^{-\alpha}\big|_{\zeta_{l-1}}^{\zeta_{l}}\nonumber \\ &- \alpha \int _{\zeta_{l-1}}^{\zeta_{l}}[g(v)-\Pi_{l}g(v)]~[\zeta_{n}-v]^{-\alpha-1}~dv\bigg]\bigg\} \nonumber \\
	=~-\frac{\alpha ~[\omega_{n}]^{-1}}{\Gamma(1-\alpha)}&\bigg\{\int_{\zeta_{0}}^{\zeta_{2}}[g(v)-\Pi_{2}g(v)]~[\zeta_{n}-v]^{-\alpha-1}~dv\nonumber \\&+\sum_{l=3}^{n}\int _{\zeta_{l-1}}^{\zeta_{l}}[g(v)-\Pi_{l}g(v)]~[\zeta_{n}-v]^{-\alpha-1}~dv \bigg\} 
	\end{align}
	Since, from (2.18) and (2.19)
	\begin{align}
	&[g(v)-\Pi_{2}g(v)]~[\zeta_{n}-v]^{-\alpha}\big|_{\zeta_{0}}^{\zeta_{2}}~=~\frac{g^{(3)}(\eta_{1})}{3!}~(v-\zeta_{0})(v-\zeta_{1})(v-\zeta_{2})[\zeta_{n}-v]^{-\alpha}\big|_{\zeta_{0}}^{\zeta_{2}}=0, \nonumber\\
	&[g(v)-\Pi_{l}g(v)]~[\zeta_{n}-v]^{-\alpha}\big|_{\zeta_{l-1}}^{\zeta_{l}}~=~\frac{g^{(4)}(\eta_{l})}{4!}~(v-\zeta_{l-3})(v-\zeta_{l-2})(v-\zeta_{l-1})(v-\zeta_{l})\big|_{\zeta_{l-1}}^{\zeta_{l}}=0.\nonumber
	\end{align}
	Consider the first integration of Eq. \eqref{2.23}
	\begin{align}
	&\bigg|\int_{\zeta_{0}}^{\zeta_{2}}[g(v)-\Pi_{2}(v)]~[\zeta_{n}-v]^{-\alpha-1}~dv\bigg| \nonumber \\&=\bigg|\int_{\zeta_{0}}^{\zeta_{2}}\frac{g^{(3)}(\eta_{1})}{3!}(v-\zeta_{0})(v-\zeta_{1})(v-\zeta_{2})[\zeta_{n}-v]^{-\alpha-1}~dv\bigg| \nonumber \\
	& \leq \frac{1}{6}~\max_{\zeta_{0} \leq \eta_{1} \leq \zeta_{2}}|g^{(3)}(\eta_{1})|(\zeta_{n}-\zeta_{2})^{-\alpha-1}\frac{(\zeta_{0}-\zeta_{2})^{3}(\zeta_{0}-2\zeta_{1}+\zeta_{2})}{12} \nonumber \\
	& \leq \frac{1}{6}~\max_{\zeta_{0} \leq \eta_{1} \leq \zeta_{2}}|g^{(3)}(\eta_{1})|(\zeta_{n}-\zeta_{2})^{-\alpha-1}\frac{(\zeta_{0}-\zeta_{2})^{3}(\zeta_{0}-\zeta_{1})}{12}. \nonumber
	\end{align}
	Using Lamma \eqref{L1}, we have following inequality
	\begin{align}\label{1st}
	\bigg|\int_{\zeta_{0}}^{\zeta_{2}}[g(v)-\Pi_{2}g(v)]~[\zeta_{n}-v]^{-\alpha-1}~dv\bigg|  \leq \frac{1}{72}\max_{t_{0} \leq x \leq t_{2}}|g^{(3)}(x)|(t_{n}-t_{2})^{-\alpha-1}L^{3-\alpha}(\tau)^{4}. \qquad
	\end{align}
	Consider the second integration of Eq. \eqref{2.23}
	\begin{align}\label{2.25}
	\bigg|&\sum_{l=3}^{n}\int _{\zeta_{l-1}}^{\zeta_{l}}[g(v)-\Pi_{l}g(v)]~[\zeta_{n}-v]^{-\alpha-1}~dv\bigg| \nonumber \\
	&\leq\bigg|\sum_{l=3}^{n-1}\int _{\zeta_{l-1}}^{\zeta_{l}}\frac{g^{(4)}(\eta_{l})}{4!}~(v-\zeta_{l-3})(v-\zeta_{l-2})(v-\zeta_{l-1})(v-\zeta_{l})[\zeta_{n}-v]^{-\alpha-1}~dv\bigg| \nonumber  \\ &+\bigg|\int _{\zeta_{n-1}}^{\zeta_{n}}\frac{g^{(4)}(\eta_{n})}{4!}~(v-\zeta_{n-3})(v-\zeta_{n-2})(v-\zeta_{n-1})(v-\zeta_{n})[\zeta_{n}-v]^{-\alpha-1}~dv  \bigg|.
	\end{align}
	Consider first part of RHS of Eq. \eqref{2.25}
	\begin{align}
	\bigg|&\sum_{l=3}^{n-1}\int _{\zeta_{l-1}}^{\zeta_{l}}\frac{g^{(4)}(\eta_{l})}{4!}~(v-\zeta_{l-3})(v-\zeta_{l-2})(v-\zeta_{l-1})(v-\zeta_{l})[\zeta_{n}-v]^{-\alpha-1}~dv\bigg| \nonumber \\
	& \leq \frac{1}{24}~\max_{\zeta_{2} \leq \eta_{2} \leq \zeta_{n-1}}|g^{(4)}(\eta_{2})|~\tilde {f}(\zeta_{l-3},\zeta_{l-2},\zeta_{l-1},\zeta_{l})\int_{\zeta_{2}}^{\zeta_{n-1}}[\zeta_{n}-v]^{-\alpha-1}~dv \nonumber \\
	& \leq \frac{1}{24\alpha}~\max_{\zeta_{2} \leq \eta_{2} \leq \zeta_{n-1}}|g^{(4)}(\eta_{2})|~\tilde{f}(\zeta_{l-3},\zeta_{l-2},\zeta_{l-1},\zeta_{l})(\zeta_{n}-\zeta_{n-1})^{-\alpha}. \nonumber
	\end{align}
	Here,  $\tilde {f}(\zeta_{l-3},\zeta_{l-2},\zeta_{l-1},\zeta_{l})$ is the $\max_{\zeta_{2} \leq v \leq \zeta_{n-1}}[(v-\zeta_{l-3})(v-\zeta_{l-2})(v-\zeta_{l-1})(v-\zeta_{l})]$. Using Lamma \eqref{L1}, we get the following inequality
	\begin{align}\label{2nd}
	\bigg|&\sum_{l=3}^{n-1}\int _{\zeta_{l-1}}^{\zeta_{l}}\frac{g^{(4)}(\eta_{l})}{4!}~(v-\zeta_{l-3})(v-\zeta_{l-2})(v-\zeta_{l-1})(v-\zeta_{l})[\zeta_{n}-v]^{-\alpha-1}~dv\bigg| \nonumber \\
	&\leq \frac{3}{128\alpha}~\max_{t_{2} \leq x_{1} \leq t_{n-1}}|g^{(4)}(x_{1})|L^{4-\alpha}(\tau)^{4-\alpha}.
	\end{align}
	\begin{remark}
		If scale function $\zeta$ fulfills the condition of Lipschitz function such that $|\zeta_{l}-\zeta_{l-1}| \leq L\tau $ then, $\max_{\zeta_{2} \leq v \leq \zeta_{n-1}}|(v-\zeta_{l-3})(v-\zeta_{l-2})(v-\zeta_{l-1})(v-\zeta_{l})|$ is obtained at $v=\zeta_{l-2}+\frac{L\tau}{2}$, therefore
		\begin{align}
		\tilde {f}(\zeta_{l-3},\zeta_{l-2},\zeta_{l-1},\zeta_{l})&=\max_{\zeta_{2} \leq v \leq \zeta_{n-1}}[(v-\zeta_{l-3})(v-\zeta_{l-2})(v-\zeta_{l-1})(v-\zeta_{l})] 
		\leq\frac{9}{16}L^{4}(\tau)^{4}.\nonumber
		\end{align}
	\end{remark}
	Consider the second part of RHS of Eq. \eqref{2.25}
	\begin{align}\label{3rd}
	&\bigg|\int _{\zeta_{n-1}}^{\zeta_{n}}\frac{g^{(4)}(\eta_{n})}{4!}~(v-\zeta_{n-3})(v-\zeta_{n-2})(v-\zeta_{n-1})(v-\zeta_{n})[\zeta_{n}-v]^{-\alpha-1}~dv  \bigg| \nonumber \\
	& \leq\frac{1}{24}~\max_{\zeta_{n-1} \leq \eta_{3} \leq \zeta_{n}}|g^{(4)}(\eta_{3})|~\bigg|\int _{\zeta_{n-1}}^{\zeta_{n}}(v-\zeta_{n-3})(v-\zeta_{n-2})(v-\zeta_{n-1})[\zeta_{n}-v]^{-\alpha}~dv  \bigg| \nonumber\\
	& \leq\frac{1}{12(1-\alpha)(2-\alpha)}\bigg(1+\frac{3}{(3-\alpha)}+\frac{3}{(3-\alpha)(4-\alpha)}\bigg)~\max_{\zeta_{n-1} \leq \eta_{3} \leq \zeta_{n}}|g^{(4)}(\eta_{3})|(\zeta_{n}-\zeta_{n-1})^{4-\alpha} \nonumber \\
	&\leq\frac{1}{12(1-\alpha)(2-\alpha)}\bigg(1+\frac{3}{(3-\alpha)}+\frac{3}{(3-\alpha)(4-\alpha)}\bigg)~\max_{t_{n-1} \leq x_{2} \leq t_{n}}|g^{(4)}(x_{2})|L^{4-\alpha}(\tau)^{4-\alpha}. \qquad
	\end{align}
	Now combining \eqref{1st}, \eqref{2nd} and \eqref{3rd}, then error bound is
	\begin{align}
	|E_{\tau}^{n}|~ \leq ~ \frac{\alpha~[\omega_{n}]^{-1}}{\Gamma(1-\alpha)}& \bigg\{\frac{1}{72}\max_{t_{0} \leq x \leq t_{2}}|g^{(3)}(x)|(t_{n}-t_{2})^{-\alpha-1}L^{3-\alpha}(\tau)^{4} \nonumber \\&+\bigg[\frac{3}{128\alpha}+\frac{1}{12(1-\alpha)(2-\alpha)}\bigg(1+\frac{3}{(3-\alpha)} \nonumber \\&+\frac{3}{(3-\alpha)(4-\alpha)}\bigg)\bigg] \max_{t_{0} \leq x_{1} \leq t_{n}}|g^{(4)}(x_{1})|L^{4-\alpha}(\tau)^{4-\alpha}\bigg\}.
	\end{align}
\end{proof}
\section{Numerical scheme for the generalized fractional advection-diffusion equation}\label{Sec-4}
In this section, we study the numerical scheme for solving the generalized fractional advection-diffusion defined by \eqref{1.1}.\\
Let $u(x,t)=\omega(t)U(x,t)-\omega(0)U_{0}(x).$~We rewrite  Eq. \eqref{1.1} to a similar equation using unity weight function and same scale function $\zeta(t)$. Then, Eq. \eqref{1.1} converts into the following form:
\begin{align}\label{e4.2}
\begin{cases}
&^{C}_{0}\mathcal{D}_{t;[\zeta(t),1]}^{\alpha}u(x,t)= D\frac{\partial^{2}u(x,t)}{\partial x^{2}}-A\frac{\partial u(x,t)}{\partial x}+f(x,t),~~ x \in \Omega,~~~ t \in (0,T],\\[1.1ex]
&u(x,0)=0, \qquad x \in  \Bar{\Omega}=\Omega \cup \partial\Omega,\\[1ex] 
&u(0,t)=\phi_{1}(t),~~u(a,t)=\phi_{2}(t),\qquad t \in (0,T],
\end{cases}
\end{align}
where $\phi_{1}(t)=\omega(t)\rho_{1}(t)-\omega(0)\rho_{1}(0)$, $\phi_{2}(t)=\omega(t)\rho_{2}(t)-\omega(0)\rho_{2}(0)$, and $f(x,t)=D\omega(0)U''_{0}(x)-A\omega(0)U'_{0}(x)+\omega(t)g(x,t).$\\[1ex] For the uniform spatial mesh, let   $0=x_{0}<x_{1}<...<x_{M}=a$ of the interval $[0,a]$ with step size $h=\frac{a}{M}$ where $M$ denotes number of subintervals, the grid points~ $x_{0}+ih~(0 \leq i \leq M),$ and $\tau=\frac{T}{N}$ be the step-size in temporal direction with grids $t_{n}=n\tau(0=t_{0}<t_{1}<...<t_{n}=T)$.\\[1.5ex]
Now, we discretize our problem \eqref{e4.2} at $(x_{i},t_{n})$, then we get 
\begin{align}\label{e4.3}
^{C}_{0}\mathcal{D}_{t;[\zeta(t),1]}^{\alpha}u(x_{i},t_{n})= Du_{xx}(x_{i},t_{n})-Au_{x}(x_{i},t_{n})+f(x_{i},t_{n}).
\end{align}
In Eq. \eqref{e4.3}, for fixed $t_{n}$ and $1 \leq i \leq M-1$,  the first and second order spatial derivatives discretized by using the following central difference approximations: 
\begin{align}\label{c}
\frac{\partial u(x_{i},t_{n})}{\partial x}= \frac{u_{n}^{i+1}-u_{n}^{i-1}}{2h}+\mathcal{O}(h^2),
\end{align}
\begin{align}\label{d}
\frac{\partial^{2} u(x_{i},t_{n})}{\partial x^{2}}= \frac{u_{n}^{i+1}-2u_{n}^{i}+u_{n}^{i-1}}{h^{2}}+\mathcal{O}(h^2).
\end{align}
With the help of Equation \eqref{ab}, we get an approximation of the generalized Caputo-type fractional derivative term in \eqref{e4.3} as follows: 
\begin{align}\label{e} ^\mathcal{H}\mathcal{D}_{t;[\zeta(t),1]}^{\alpha}u(x_{i},t_{n})=&\lambda_{0}u(x_{i},t_{n})+\lambda_{1}u(x_{i},t_{n-1})+\lambda_{2}u(x_{i},t_{n-2})+\sum_{l=3}^{n-3}\lambda_{n-l}~u(x_{i},t_{l}) \\&+\lambda_{n-2}u(x_{i},t_{2}) \nonumber +\lambda_{n-1}u(x_{i},t_{1})+\lambda_{n}u(x_{i},t_{0})+\mathcal{O}(\tau ^{4-\alpha}).
\end{align}
where $\lambda_{l}$ defined in Eq. \eqref{3.17}. 
Next, we use Eq. \eqref{c}, Eq. \eqref{d}, and Eq. \eqref{e} in discretized equation \eqref{e4.3}, yields
\begin{align}\label{e4.7}
\sum_{l=0}^{n}\lambda_{n-l}~u(x_{i},t_{l})=Du_{xx}(x_{i},t_{n})-Au_{x}(x_{i},t_{n})+f_{n}^{i}+e_{n}^{i}.
\end{align} 
where, $|e_{n}^{i}| \leq \tilde{c}(\tau^{4-\alpha}+h^{2})$ for some constant $\tilde{c}.$\\
Now, we use numerical approximation $u_{n}^{i}$ of $u(x_{i},t_{n})$ to neglect the truncation error term in Equation \eqref{e4.7}, then we determine the following finite difference scheme:
\begin{align}\label{4.7}
\lambda_{0}u_{n}^{i}+\lambda_{1}u_{n-1}^{i}+\sum_{l=3}^{n-2}\lambda_{n-l}u_{l}^{i}+\lambda_{n-2}u_{2}^{i}\nonumber +\lambda_{n-1}u_{1}^{i}+\lambda_{n}u_{0}^{i}&=D~\frac{u_{n}^{i+1}-2u_{n}^{i}+u_{n}^{i-1}}{h^{2}}\\&-A~\frac{u_{n}^{i+1}-u_{n}^{i-1}}{2h}+f_{n}^{i} ~,
\end{align} 
where
$1 \leq n \leq N$ , ~$1 \leq i \leq M-1$. That is,
\begin{align}\label{sch}
\begin{cases}	(\frac{D}{h^{2}}+\frac{A}{2h})u_{1}^{i-1}-(\lambda_{0}+\frac{2D}{h^{2}})u_{1}^{i}+(\frac{D}{h^{2}}-\frac{A}{2h})u_{1}^{i+1}=&\lambda_{1}u_{0}^{i}-f_{1}^{i},~~n=1,\\[1.5ex]
(\frac{D}{h^{2}}+\frac{A}{2h})u_{2}^{i-1}-(\lambda_{0}+\frac{2D}{h^{2}})u_{2}^{i}+(\frac{D}{h^{2}}-\frac{A}{2h})u_{2}^{i+1}=&\lambda_{1}u_{1}^{i}+\lambda_{2}u_{0}^{i}-f_{2}^{i},~~n=2,\\[2ex](\frac{D}{h^{2}}+\frac{A}{2h})u_{n}^{i-1}-(\lambda_{0}+\frac{2D}{h^{2}})u_{n}^{i}+(\frac{D}{h^{2}}-\frac{A}{2h})u_{n}^{i+1}=&\lambda_{1}u_{n-1}^{i}+\sum_{l=0}^{n-2}\lambda_{n-l}u_{l}^{i}-f_{n}^{i},~~n \geq 3. \\[1.5ex]
u_{0}^{i}=0, ~~\qquad 0 \leq i \leq M,\\[1ex]
u_{n}^{0}=\phi_{1}(t_{n}),~~u_{n}^{M}=\phi_{2}(t_{n}),~\qquad~0 \leq n \leq N,
\end{cases}
\end{align}
We rewrite the matrix form of above equation as follows:
\begin{align}\label{f}
\begin{cases}
KU_{1}=\lambda_{1}U_{0}-F_{1}+H_{1},~~n=1,\\[1ex] 
KU_{2}=\lambda_{1}U_{1}+\lambda_{2}U_{0}-F_{2}+H_{2},~~n=2,\\[1ex]
KU_{n}=\lambda_{1}U_{n-1}+\sum_{l=0}^{n-2}\lambda_{n-l}U_{l}-F_{n}+H_{n},~~n \geq 3,
\end{cases}
\end{align}
where the matrix as well as vectors of Equation \eqref{f} are defined as follows: \\[1ex]
$ K= tri~\bigg[\frac{D}{h^{2}}+\frac{A}{2h},~-\lambda_{0}-\frac{2D}{h^{2}},~\frac{D}{h^{2}}-\frac{A}{2h}\bigg]_{(M-1)\times(M-1)},\\[1.5ex]
U_{n}= (u_{n}^{1},u_{n}^{2},...,u_{n}^{M-1})^{T},\\[1.5ex]
F_{n}= (f_{n}^{1},f_{n}^{2},...,f_{n}^{M-1})^{T},\\[1.5ex]
H_{n}=\bigg((-\frac{D}{h^{2}}-\frac{A}{2h})u_{n}^{0},0,...,0,(-\frac{D}{h^{2}}+\frac{A}{2h})u_{n}^{M}\bigg)^{T},~~1\leq n \leq N.$
\begin{remark}
	Since the coefficient matrix $A$ is a tridiagonal and strictly diagonally dominant then $det(A) \ne 0$ (Levy-Desplanques theorem). Therefore at each time level $t_{n},$  the proposed scheme \eqref{sch} has an unique solution for $ 1 \leq n \leq N.$  
\end{remark}
\begin{Theorem}\label{2}
	The local truncation error of difference scheme \eqref{4.7} at $(x_{i},t_{n}),~ 1 \leq i \leq M-1,~ 1 \leq n \leq N$ is
	\begin{align}
	|e_{n}^{i}| \leq C (h^{2}+\tau^{4-\alpha}),
	\end{align}
	where $C$ is the positive constant independent of the time and space step sizes.
 \end{Theorem}
\begin{proof}
From Eqs.~(4.7) the LTE of difference scheme \eqref{4.7} is
\begin{align}
e_{n}^{i}&=~\lambda_{0}u(x_{i},t_{n})+\lambda_{1}u(x_{i},t_{n-1})+\lambda_{2}u(x_{i},t_{n-2})+\sum_{l=3}^{n-3}\lambda_{n-l}u(x_{i},t_{l}) +\lambda_{n-2}u(x_{i},t_{2}) \nonumber \\&+\lambda_{n-1}u(x_{i},t_{1})+\lambda_{n}u(x_{i},t_{0})-D~ \frac{u_{n}^{i+1}-2u_{n}^{i}+u_{n}^{i-1}}{h^{2}}+A~\frac{u_{n}^{i+1}-u_{n}^{i-1}}{2h}-f_{n}^{i} \nonumber\\
&=\bigg[\lambda_{0}u(x_{i},t_{n})+\lambda_{1}u(x_{i},t_{n-1})+\lambda_{2}u(x_{i},t_{n-2})+\sum_{l=2}^{n-3}\lambda_{n-l}u(x_{i},t_{l}) +\lambda_{n-1}u(x_{i},t_{1})\nonumber\\&+\lambda_{n}u(x_{i},t_{0})-^\mathcal{H}\mathcal{D}_{t;[\zeta(t),1]}^{\alpha}u(x_{i},t_{n})\bigg] -D\bigg[\frac{u_{n}^{i+1}-2u_{n}^{i}+u_{n}^{i-1}}{h^{2}}-\frac{\partial^{2} u(x_{i},t_{n})}{\partial x^{2}}\bigg]\nonumber \\&+A\bigg[\frac{u_{n}^{i+1}-u_{n}^{i-1}}{2h}-\frac{\partial u(x_{i},t_{n})}{\partial x}\bigg]   \nonumber  \\
&=\mathcal{O}(\tau^{4-\alpha})-D~\mathcal{O}(h^{2})+A~\mathcal{O}(h^{2})=\mathcal{O}(\tau^{4-\alpha}+h^{2}).\nonumber
\end{align}
\end{proof}
In next theorem we use $L^{2}(\Omega)$-space with the norm $\|.\|_{2}$ and the inner product $\langle . , . \rangle $.
\begin{Theorem}(see \cite{xu2013numerical})
If tridiagonal matrix elements satisfy the inequality
\begin{align}\label{3.12}
\lambda_{1} \leq(\lambda_{0})^{M-3} \bigg(\lambda_{0}+\frac{D}{h^{2}}-\frac{A}{2h}\bigg)\bigg(\lambda_{0}+\frac{D}{h^{2}}+\frac{A}{2h}\bigg),
\end{align}
then finite difference scheme is stable.
\end{Theorem}
\begin{proof}
The Eq. \eqref{f} can be rewritten as following, for $1 \leq n \leq N$
\begin{align}
KU_{n}=\lambda_{1}U_{n-1}+\sum_{l=0}^{n-2}\lambda_{n-l}U_{l}-F_{n}+H_{n},     \nonumber   
\end{align} 
let,~~~~~~~~~~~~~~~~~~~~~~~~~~~
$
V_{n}=\sum_{l=0}^{n-2}\lambda_{n-l}U_{l}-F_{n}+H_{n}.    \nonumber
$ \\[1.5ex]
Since matrix $K$ is invertible then above equation can be rewritten as 
\begin{align}
U_{n}=\lambda_{1}K^{-1}U_{n-1}+K^{-1}V_{n}. \nonumber
\end{align} 
Using recurrence relation, we get the following equation
\begin{align}
U_{n}=&(\lambda_{1}K^{-1})^{2}U_{n-2}+(\lambda_{1}K^{-1})K^{-1}V_{n-1}+K^{-1}V_{n}  \nonumber \\[0.5ex] 
=&(\lambda_{1}K^{-1})^{n}U_{0}+(\lambda_{1}K^{-1})^{n-1}K^{-1}V_{1}+(\lambda_{1}K^{-1})^{n-2}K^{-1}V_{2}+...+K^{-1}V_{n}.
\end{align}
Let $ \tilde{U}_{n}$ is the approximate solution of Eq. \eqref{f}, then we define error at grid point $t_{n}=n\tau$  
$$e_{n}=U_{n}-\tilde{U}_{n}, ~~~~~~0 \leq n \leq N.$$
Then, we obtain 
$$e_{n}=(\lambda_{1}K^{-1})^{n}e_{0},~~~~~1 \leq n \leq N.$$
From definition of compatible matrix norm
$$\|e_{n}\| ~\leq ~\|(\lambda_{1}K^{-1})^{n}\| ~\|e_{0}\|.$$
Since,
\begin{align}
\|(\lambda_{1}K^{-1})^{n}\| \leq \|(\lambda_{1}K^{-1})\|~\|(\lambda_{1}K^{-1})^{n-1}\| \leq ... \leq \|(\lambda_{1}K^{-1})\|^{n} \nonumber
\end{align} 
According to Ostrowski therorem (\cite{Xu2014}, Theorem 3.1), let $K=[a_{i,j}]_{{M-1}\times{M-1}}$, then determinant of $K$ satisfies
\begin{align}\label{3.14}
|det(K)|~ &\geq \prod_{i=1}^{M-1}\Bigg(|a_{i,i}|-\sum_{{j=1},{j \ne i}}^{M-1}|a_{i,j}|\Bigg) \nonumber \\[1ex] &=(\lambda_{0})^{M-3} \bigg(\lambda_{0}+\frac{D}{h^{2}}-\frac{A}{2h}\bigg)\bigg(\lambda_{0}+\frac{D}{h^{2}}+\frac{A}{2h}\bigg).
\end{align}
Using assumed inequality \eqref{3.12} into \eqref{3.14}, we get
$
|det(K^{-1})| \leq \frac{1}{\lambda_{1}}, $\\[1ex] it implies that,~~~~~~~~~~~~~ 
$$|\lambda_{1}|~\|K^{-1}\| \leq 1. \nonumber$$\\[1ex] 
Therefore,~~~~~~~~~~~~~ ~~~~~~~~~~~\qquad\qquad $
|e_{n}| \leq ~ |e_{0}|. $ \\[1ex]
Thus, the numerical scheme is stable.
\end{proof}
\begin{Theorem} \label{4} 
The solution $U_{n}^{i}$ of the difference scheme \eqref{4.7} satisfies 
\begin{align}
\max_{i,n}|U(x_{i},t_{n})-\tilde{U}(x_{i},t_{n})| \leq {C}(h^{2}+\tau^{4-\alpha}) 
\end{align}
for some constant $C.$
\end{Theorem}
\begin{proof}
The truncation error for difference scheme at $(x_{i},t_{n}) \in [0,a] \times [0,T]$ is \\
$$|e_{n}^{i}| \leq C(h^{2}+\tau^{4-\alpha})$$ by Theorem \ref{2}. Now using theorem of stability, it implies
\begin{align}
\max_{i,n}|U(x_{i},t_{n})-\tilde{U}(x_{i},t_{n})| \leq C |U(x_{i},t_{0})-\tilde{U}(x_{i},t_{0})|
\end{align}
We obtain the desired result easily after using the truncation error as discussed in Theorem \ref{2}.   
\end{proof}
\section{Numerical results} \label{Sec-5}
In this section, we will check the numerical accuracy of the proposed schemes \eqref{ab} and difference scheme \eqref{sch}, also verify the theoretical convergence order discussed in Theorem \eqref{4}. Here, we provide three examples to numerically support our theory; in the first example we check convergence order and absolute error of approximation for the GCFD, while last two problems get the form \eqref{e4.2} to discribe accuracy and maximum absolute error of the difference scheme. We provide two Tables \eqref{7} and \eqref{10} to compare our scheme \eqref{sch} with Gao et al. (\cite{gao2014new}, Example 4.1) and  Cao et al. (\cite{cao2015high}, Example 5.1), respectively for particular choice of scale $\zeta(t)=t$ and weight $\omega(t)=1$ functions. All numerical results are implemented in MATLAB R2018b. \\
To calculate the maximum absolute error $E_{\infty}$ and error $E_{2}$ corresponding to the $L_{2}$-norm, we use following formulas, respectively.  
\begin{align}
E_{\infty}(M,N)~=~ \max_{1\leq i \leq M-1}|U_{N}^{i}-u_{N}^{i}|, \nonumber \\[1ex]
E_{2}(M,N)~=~\bigg(h\sum_{i=1}^{M-1}|U_{N}^{i}-u_{N}^{i}|^{2}\bigg)^{1/2}, \nonumber
\end{align}
where $\{U_{n}^{i}\}$ is the exact solution of advection diffusion equation and $\{u_{n}^{i}\}$ is the approximate solution at the point $(x_{i},t_{n}).$ \\
Moreover, the convergence order in space and time direction for the described difference scheme corresponding to $L_{\infty}$-norm can be evaluated using the following formulas. $R_{x}$ is the order of convergence in space side and $R_{t}$ for temporal side.
$$R_{x}~=~\frac{log(E(2M,N))-log(E(M,N))}{log(2)},$$\\ and $$R_{t}~=~\frac{log(E(M,2N))-log(E(M,N))}{log(2)}.$$ \nonumber
\begin{Example}\label{5.1} \cite{gao2014new}
Take function $u(t)=t^{4+\alpha}$,~ $t \in [0,1],$~ for~ $0< \alpha <1.$ Determine the $\alpha$-th order GCFD for $u(t)$ at $T=1$ numerically.
\end{Example}
The maximum absolute error and rate of convergence for the scheme \eqref{ab} to approximate the GCFD of function $u(t)$ for $\alpha = 0.2, 0.5, 0.8$ with uniform time steps    $1/10$, $1/20$, $1/40$, $1/80$, $1/160$ are calculated and shown in following tables.\\
Table \ref{1} shows that the errors with respect to $L_{\infty}$-norm and convergence rates in time, and these data are found after calculating the classical Caputo derivative (i.e. taking $\zeta(t)=t,$ $\omega(t)=1$) of function $u(t)$ with the help of scheme \eqref{ab}. From this Table, we can see that the errors of our scheme \eqref{ab} obtaining from GCFD approximation are lesser in compare to scheme developed in Gao et al. \cite{gao2014new} for approximation of Caputo derivative and convergence rate of our scheme is $(4-\alpha),$ while accuracy for time derivative in \cite{gao2014new} is $(3-\alpha).$ 
In Table \ref{2}, to compute the maximum errors $E_{\infty}$ and order of convergence in time direction, we take $\omega(t)=e^{t}$ while scale function is fixed with $t$. From Table \ref{3}, we validate the the convergence for $\zeta(t)=t,$ ~$\omega(t)=t+1$ and the CPU time in seconds for $\alpha=0.8$ are discussed. Table \ref{4} shows that maximum errors and order of convergence for different choice of weight $\omega(t)=t^{0.5},~ t,~ t^{4}$,~$e^{2t}$,~ scale is $\zeta(t)=t$ and $\alpha=1/3$. In all cases, accuracy in time is obtained as  $(4-\alpha)$ for the scheme \eqref{ab}, which is higher than \cite{xu2014numerical,yadav2019high}.  
\begin{table}[H]
	\centering
	\caption{$E_{\infty}$ errors and convergence rates $R_{t}$ for Example \ref{5.1}, when  $\zeta(t)=t,$ $\omega(t)=1$ and different $ \alpha$'s.}\label{1}
	\begin{tabular}{ccccccc}
		\hline
		& ~~~~~~~~$\alpha=0.2$ &                  & ~~~~~~~~$\alpha=0.5$ &                             & ~~~~~~~~$\alpha=0.5$ \cite{gao2014new} \\ \cline{1-7}
		$N$   & $E_{\infty}$ error     & $R_{t}$       & $E_{\infty}$ error    & $R_{t}$       & $E_{\infty}$ error     & $R_{t}$       \\ \hline
		10  & 1.6978e-04       &               & 1.5401e-03       &               & 1.3507e-02       &               \\
		20  & 1.3130e-05       & 3.6928    & 1.4383e-04       & 3.4206    & 2.6121e-03       & 2.3704    \\
		40  & 9.9792e-07       & 3.7178    & 1.3116e-05       & 3.4550    & 4.8618e-04       & 2.4256    \\
		80  & 7.4966e-08       & 3.7346    & 1.1811e-06       & 3.4731    & 8.8645e-05       & 2.4554    \\
		160 & 5.5944e-09       & 3.7442     & 1.0560e-07       & 3.4835    & 1.5975e-05       & 2.4722 \\ \hline
	\end{tabular}
\end{table}
\begin{table}[H]
	\centering
	\caption{$E_{\infty}$ errors and convergence rates $R_{t}$ for Example \ref{5.1}, when $\zeta(t)=t,$ $\omega(t)=e^{t}$ and different $ \alpha $'s.}\label{2}
	\begin{tabular}{cccccccc}
		\hline
		&~~ $\alpha=0.2$ &            &~~ $\alpha=0.5$ &            & ~~$\alpha=0.8$ &            & ~~$\alpha=0.8$ \\ \cline{1-8}
		$N$   & $E_{\infty}$ error     & $R_{t}$       & $E_{\infty}$ error     & $R_{t}$       & $E_{\infty}$ error     & $R_{t}$       & CPU time (s) \\ \hline
		10  & 7.5552e-04       &               & 6.3075e-03       &               & 3.2184e-02       &               & 11.354861    \\
		20  & 7.1075e-05       & 3.4101    & 6.7873e-04       & 3.2162    & 4.2148e-03       & 2.9328    & 33.306962    \\
		40  & 5.9370e-06       & 3.5815    & 6.6677e-05       & 3.3476    & 5.0375e-04       & 3.0647    & 82.772072    \\
		80  & 4.6934e-07       & 3.6610    & 6.2491e-06       & 3.4155    & 5.7497e-05       & 3.1312    & 180.064565   \\
		160 & 3.5650e-08       & 3.7186    & 5.7027e-07       & 3.4539    & 6.4090e-06       & 3.1653    & 419.978865 \\ \hline
	\end{tabular}
\end{table}
\begin{table}[H]
	\centering
	\caption{$E_{\infty}$ errors and convergence rates
		$R_{t}$ for Example \ref{5.1}, when $\zeta(t)=t,$ $\omega(t)=t+1$ and different $ \alpha $'s.}\label{3}
	\begin{tabular}{cccccccc}
	\hline
	&~~ $\alpha=0.2$ &            &~~ $\alpha=0.5$ &            & ~~$\alpha=0.8$ &            & ~~$\alpha=0.8$ \\ \cline{1-8}
	$N$   & $E_{\infty}$ error     & $R_{t}$       & $E_{\infty}$ error     & $R_{t}$       & $E_{\infty}$ error     & $R_{t}$       & CPU time (s) \\ \hline
	10  & 3.5019e-04       &               & 3.1752e-03       &               & 1.7251e-02       &               & 10.194270    \\
	20  & 3.0621e-05       & 3.5155    & 3.1426e-04       & 3.3368    & 2.0904e-03       & 3.0448    & 25.931157    \\
	40  & 2.4466e-06       & 3.6457    & 2.9499e-05       & 3.4132    & 2.3982e-04       & 3.1238    & 65.689186    \\
	80  & 1.8844e-07       & 3.6986    & 2.6976e-06       & 3.4509    & 2.6799e-05       & 3.1617    & 155.623905   \\
	160 & 1.4248e-08       & 3.7252    & 2.4326e-07       & 3.4711    & 2.9560e-06       & 3.1805    & 384.953845   	\\ \hline
\end{tabular}
\end{table}
\begin{table}[H]
\centering
\caption{$E_{\infty}$ errors and convergence rates $R_{t}$ for Example \ref{5.1}, when $\zeta(t)=t,$ $\alpha=\frac{1}{3}$ and different weight functions.}\label{4}
\begin{tabular}{cccccccc}
\hline
&~~ $\omega(t)=t^{0.5}$ &            &~~~ $\omega(t)=t$ &            & ~~~$\omega(t)=t^{4}$ &            & ~~~$\omega(t)=e^{2t}$ \\ \cline{1-8}
$N$ & $E_{\infty}$ error                 & $R_{t}$         & $E_{\infty}$ error           & $R_{t}$         & $E_{\infty}$ error               & $R_{t}$         & $R_{t}$       \\ \hline
10       & 9.5450e-04          &            & 1.7465e-03    &            & 3.2557e-02        &            &             \\
20       & 8.4374e-05          & 3.4999 & 1.5027e-04    & 3.5388 & 2.0348e-03        & 4.0000 & 3.2332            \\
40       & 7.0877e-06          & 3.5734 & 1.2845e-05    & 3.5483 & 1.2749e-04        & 3.9964 & 3.4225           \\
80       & 5.8166e-07          & 3.6071 & 1.0650e-06    & 3.5923 & 1.1169e-05        & 3.5129 & 3.5220           \\
160      & 4.7123e-08          & 3.6257 & 8.6813e-08    & 3.6167 & 9.3941e-07        & 3.5716 & 3.5764	\\ \hline
\end{tabular}
\end{table}
\begin{figure}[h]
	\centering
	\begin{subfigure}{.5\textwidth}
		\centering
		\includegraphics[width=7.5cm,height=6.4cm]{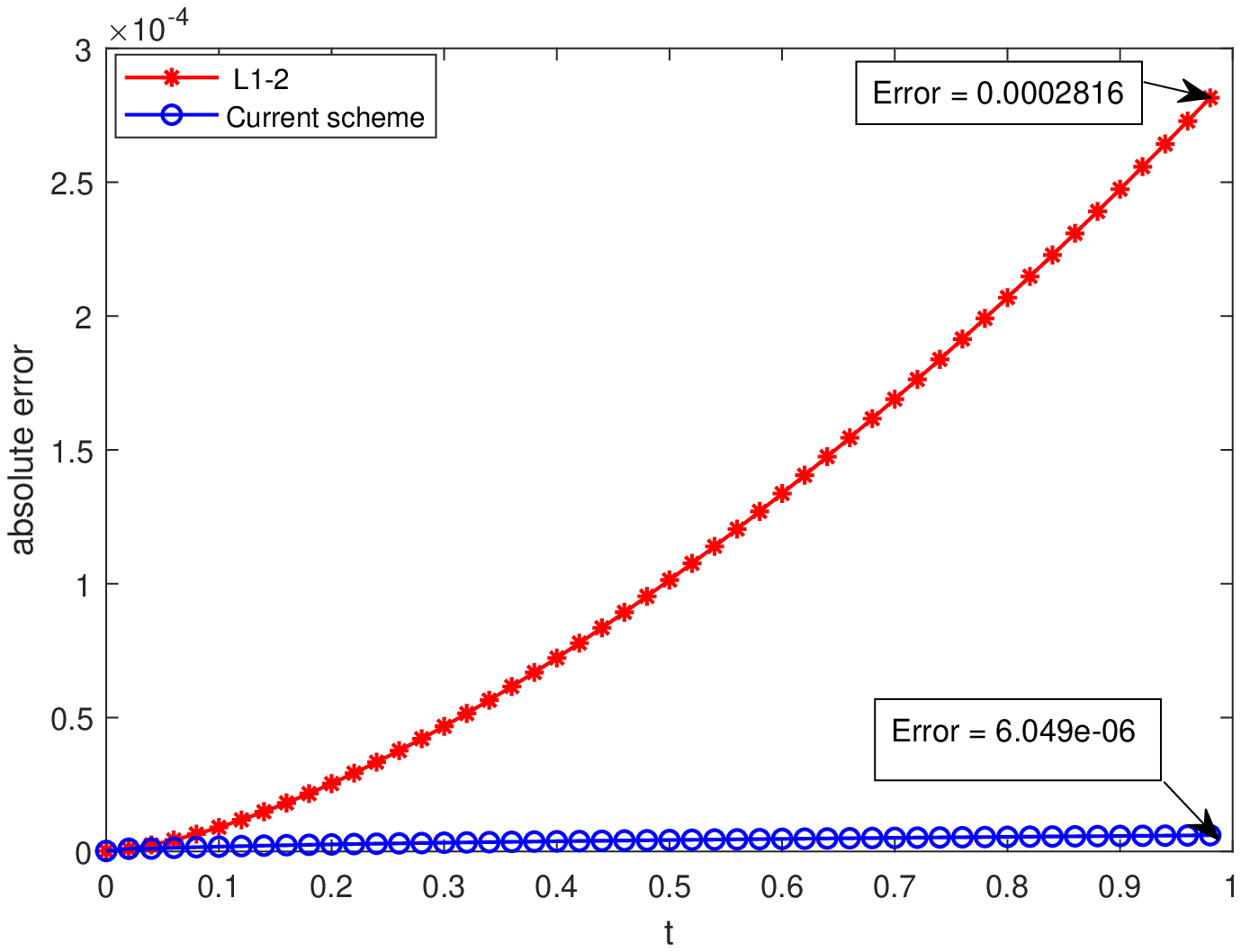}\label{EX1}
		\caption{}
		\label{EX1}
	\end{subfigure}%
	\begin{subfigure}{.5\textwidth}
		\centering
		\includegraphics[width=7.5cm,height=6.4cm]{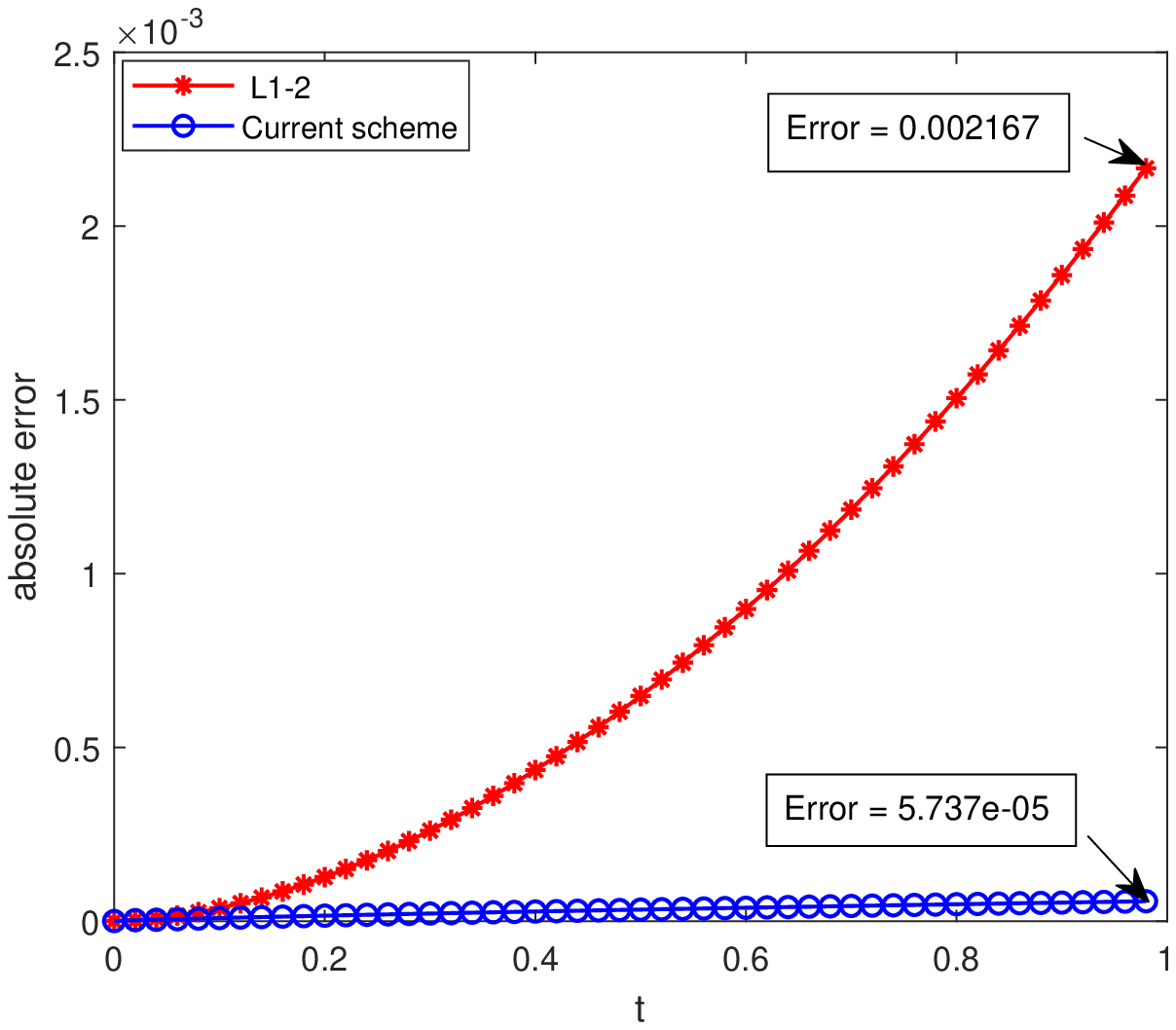}\label{EX2}
		\caption{}
		\label{Ex2}
	\end{subfigure}
	\caption{Errors plot of the numerical results for the Example \eqref{5.1} at final time $T=1$ for different values of $\alpha$ (Left side \eqref{EX1} for $\alpha=0.5$ ; Right side  \eqref{Ex2} for $\alpha=0.8$).}
\end{figure} 
\begin{Example}\label{5.2}
\cite{cao2015high} We take the following generalized fractional advection-diffusion equation: 
\begin{align}
\begin{cases}
^{C}_{0}\mathcal{D}_{t;[\zeta(t),\omega(t)]}^{\alpha}u(x,t)~=~\frac{\partial^{2} u(x,t)}{\partial x^{2}}-\frac{\partial u(x,t)}{\partial x}+f(x,t),~~(x,t) \in (0,1) \times (0,1), \\
u(x,0)=0,~~~~ x \in (0,1),\\
u(0,t)= t^{6+\alpha},~ u(1,t)= et^{6+\alpha},~~~ t \in (0,1],
\end{cases}
\end{align}
where $f(x,t)=e^{x}t^{6}\frac{\Gamma(7+\alpha)}{720}$.
When $\zeta(t)=t$ and $\omega(t)=1$, then  $u(x,t)=e^{x}t^{6+\alpha}$ is the exact solution. 
\end{Example}
To solve this example, we use the numberical scheme defined in \eqref{sch}. The maximum errors at time $t=1$ for different values of $\alpha$ with different step sizes, and rate of convergence in time direction and space direction are displayed in Tables \ref{5} and \ref{6}. In Table \ref{5}, we set $h=\frac{1}{2000}$ and describe the numerical errors and convergence rates in time $R_{t}$ for different values of $N$. In Table \ref{6}, we fix $\tau=\frac{1}{500}$ and present the numerical errors and spatial convergence rates $R_{x}$ for different values of $M$. It is shown that our scheme \eqref{sch} gives $(4-\alpha)$-order convergence in temporal direction and second-order convergence in spatial direction.
\begin{figure}[H]
\centering
\begin{subfigure}{.5\textwidth}
\centering
\includegraphics[width=7.5cm,height=6.4cm]{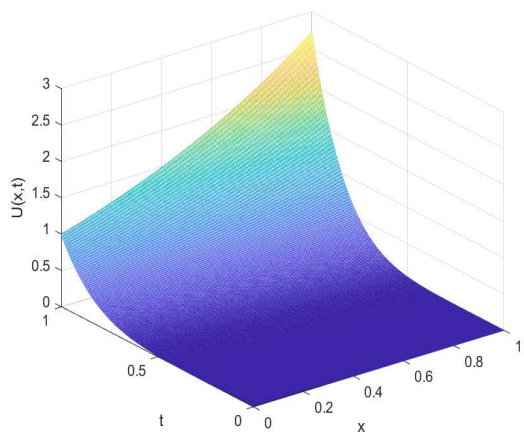}
\caption{Exact solution for $\alpha=0.95$}
\label{EX5:subfig1}
\end{subfigure}%
\begin{subfigure}{.5\textwidth}
\centering
\includegraphics[width=7.5cm,height=6.4cm]{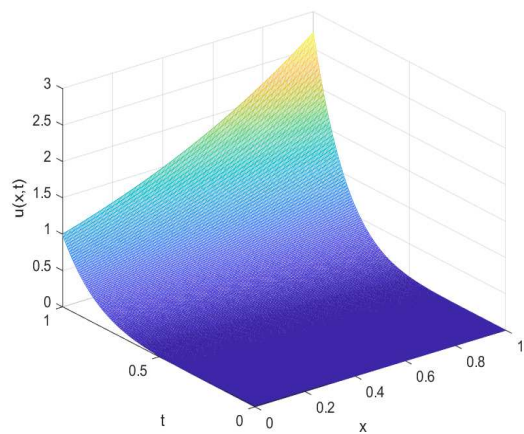}
\caption{Numerical solution for $\alpha=0.95$}
\label{Ex5:subfig2}
\end{subfigure}
\newline
\begin{subfigure}{.5\textwidth}
\centering
\includegraphics[width=7cm,height=6cm]{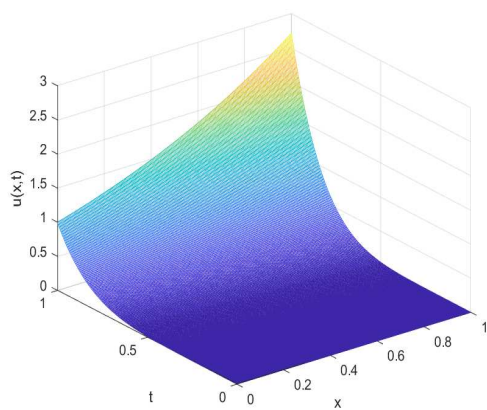}
\caption{Numerical solution for $\alpha=0.6$}
\label{Ex5:subfig3}
\end{subfigure}%
\begin{subfigure}{.5\textwidth}
\centering
\includegraphics[width=7cm,height=6cm]{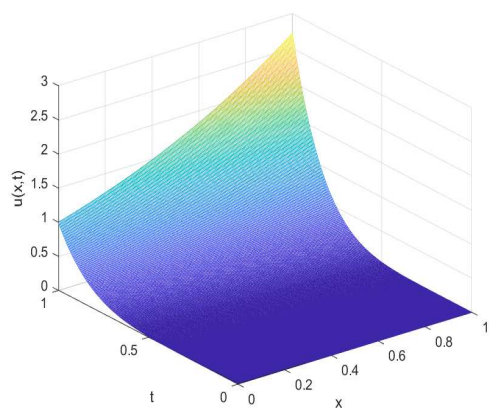}
\caption{Numerical solution for $\alpha=0.45$}
\label{Ex5:subfig4}
\end{subfigure}
\caption{Exact and approximate solutions of Example \ref{5.2} for different $\alpha$'s with $M=N=200$ and $\omega(t)=1$, $\zeta(t)=t$.}
\label{Ex5:fig1}
\end{figure} 
\begin{table}[H]
\centering
\caption{Errors $E_{\infty}$ and $E_{2}$ with convergence rates in time for Example \ref{5.2}, when $h=1/2000$ and different $\alpha$'s.}\label{5}
\begin{tabular}{llllllllllllllllllllllll}  \hline
\multicolumn{4}{l}{$\alpha$} & \multicolumn{4}{l}{$N$} & \multicolumn{4}{l}{~~$E_{\infty}$  error} & \multicolumn{4}{l}{~$R_{t}$} & \multicolumn{4}{l}{~~$E_{2}$  error}  & \multicolumn{4}{l}{~$R_{t}$} \\ \hline
0.8     &      &      &      & 8      &     &     &    & 1.6385e-02     &     &     &     &           &     &     &     & 7.4195e-03   &    &    &    &              &    &    &    \\
&      &      &      & 16     &     &     &    & 2.2920e-03     &     &     &     & 2.8377    &     &     &     & 1.3514e-03   &    &    &    & 2.4569   &    &    &    \\
&      &      &      & 32     &     &     &    & 2.8191e-04     &     &     &     & 3.0233    &     &     &     & 1.8556e-04   &    &    &    & 2.8645   &    &    &    \\
&      &      &      & 64     &     &     &    & 3.2609e-05     &     &     &     & 3.1119    &     &     &     & 2.2573e-05   &    &    &    & 3.0392   &    &    &    \\
&      &      &      & 128    &     &     &    & 3.6574e-06     &     &     &     & 3.1564    &     &     &     & 2.5937e-06   &    &    &    & 3.1215   &    &    &    \\
0.5     &      &      &      & 8      &     &     &    & 3.7940e-03     &     &     &     &           &     &     &     & 1.8168e-03   &    &    &    &              &    &    &    \\
&      &      &      & 16     &     &     &    & 4.2889e-04     &     &     &     & 3.1451    &     &     &     & 2.5900e-04   &    &    &    & 2.8104   &    &    &    \\
&      &      &      & 32     &     &     &    & 4.3064e-05     &     &     &     & 3.3160    &     &     &     & 2.8647e-05   &    &    &    & 3.1765   &    &    &    \\
&      &      &      & 64     &     &     &    & 4.0768e-06     &     &     &     & 3.4010    &     &     &     & 2.8348e-06   &    &    &    & 3.3371   &    &    &    \\
&      &      &      & 128    &     &     &    & 3.7173e-07     &     &     &     & 3.4551    &     &     &     & 2.6407e-07   &    &    &    & 3.4243   &    &    &    \\
0.2     &      &      &      & 8      &     &     &    & 5.4498e-04     &     &     &     &           &     &     &     & 2.7293e-04   &    &    &    &              &    &    &    \\
&      &      &      & 16     &     &     &    & 5.1093e-05     &     &     &     & 3.4150    &     &     &     & 3.1455e-05   &    &    &    & 3.1172   &    &    &    \\
&      &      &      & 32     &     &     &    & 4.2949e-06     &     &     &     & 3.5724    &     &     &     & 2.8822e-06   &    &    &    & 3.4480   &    &    &    \\
&      &      &      & 64     &     &     &    & 3.3828e-07     &     &     &     & 3.6663    &     &     &     & 2.3627e-07   &    &    &    & 3.6087   &    &    &    \\
&      &      &      & 128    &     &     &    & 2.2628e-08     &     &     &     & 3.9020    &     &     &     & 1.6173e-08   &    &    &    & 3.8687   &    &    &   \\ \hline
\end{tabular} 
\end{table}
\begin{figure}[H]
\begin{subfigure}{.5\textwidth}
\includegraphics[width=8cm,height=6.4cm]{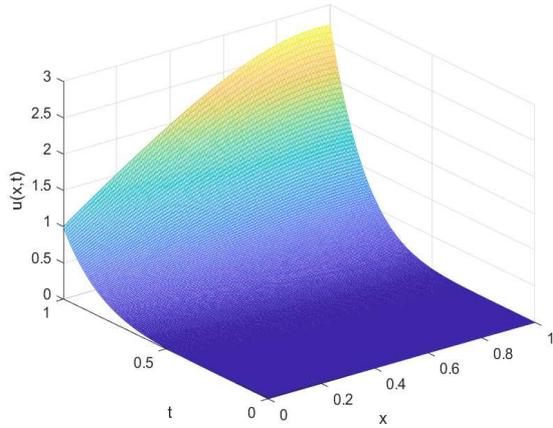}
\caption{$\zeta(t)=e^{t},~ \omega(t)=1$}
\label{Ex5:subfig2}
\end{subfigure}
\begin{subfigure}{.5\textwidth}
\includegraphics[width=8cm,height=6.4cm]{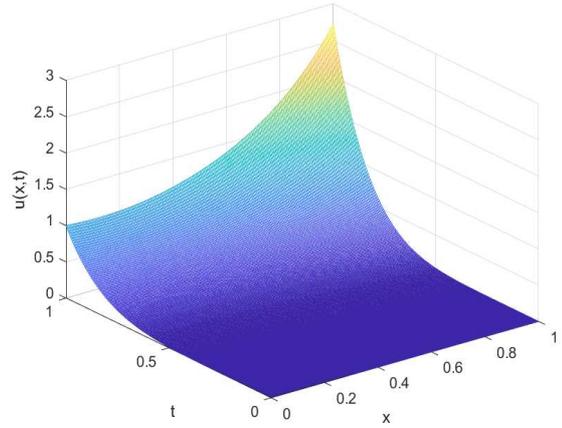}
\caption{$\zeta(t)=t^{1/2},~ \omega(t)=1$}
\label{Ex5:subfig3}
\end{subfigure}
\caption{Numerical solutions of Example \ref{5.2} with different choices of weight functions $\omega(t)$ and scale functions $\zeta(t)$, and $M=N=200$ with $\alpha=0.9$.}
\label{Ex5:fig1}
\end{figure}
Clearly visible from above figures that scale function $\zeta(t)$ can stretch or contract the domain.
\begin{table}[H]
\centering
\caption{The maximum errors and convergence rates $R_{x}$ for Example \ref{5.2}, when $\tau=1/500$ and different $\alpha$'s.}\label{6}
\begin{tabular}{cccccccc}
\hline
& \qquad~~~$\alpha=0.2$ &            &\qquad~~~$\alpha=0.5$ &            &\qquad ~~~$\alpha=0.8$ &            \\ \cline{1-7}  
$M$ & $E_{\infty}$ error          & $R_{x}$         & $E_{\infty}$ error          & $R_{x}$         & $E_{\infty}$ error          & $R_{x}$         \\ \hline
8        & 2.3437e-04   &            & 2.1275e-04   &            & 1.8078e-04   &            \\
16       & 5.8814e-05   & 1.9945 & 5.3334e-05   & 1.9960 & 4.5242e-05   & 1.9985 \\
32       & 1.4733e-05   & 1.9971 & 1.3373e-05   & 1.9958 & 1.1319e-05   & 1.9990 \\
64       & 3.6832e-06   & 2.0001 & 3.3408e-06   & 2.0010 & 2.7940e-06   & 2.0183 \\
128      & 9.2088e-07   & 1.9999 & 8.3276e-07   & 2.0042 & 6.6258e-07   & 2.0762 \\ \hline
\end{tabular}
\end{table}	
\begin{table}[H]
\centering
\caption{The maximum errors and convergence rates $R_{t}$ of Example 4.1 in \cite{gao2014new} (page no 43) for different values of $\alpha$ with $\tau=1/2000.$}\label{7}
\begin{tabular}{cccccc} \hline
&     &\qquad\qquad Current scheme & &\qquad\qquad~~~$L1-2$ \cite{gao2014new} \\ \hline
$\alpha$ &   $N$ & $E_{\infty}$      & $R_{t}$        & $E_{\infty}$   & $R_{t}$     \\ \hline
0.9      & 10  & 2.2536e-03        &                & 1.8600e-02     &             \\
& 20  & 3.2727e-04        & 2.7837     & 4.7228e-03     & 1.9776  \\
& 40  & 4.1906e-05        & 2.9653     & 1.1488e-03     & 2.0395  \\
& 80  & 5.1072e-06        & 3.0366     & 2.7367e-04     & 2.0697  \\
& 160 & 6.1092e-07        & 3.0635     & 6.4520e-05     & 2.0846  \\
0.5      & 10  & 2.5470e-04        &                & 2.4936e-03     &             \\
& 20  & 2.6005e-05        & 3.2919     & 4.8366e-04     & 2.3662  \\
& 40  & 2.4615e-06        & 3.4012     & 9.0163e-05     & 2.4234  \\
& 80  & 2.2827e-07        & 3.4307     & 1.6457e-05     & 2.4539  \\
& 160 & 2.3662e-08        & 3.2701     & 2.9700e-06     & 2.4701  \\ \hline
\end{tabular}
\end{table}
\begin{Example} \label{5.4}
Take the following fractional advection-diffusion equation:
\begin{align}
\begin{cases}
^{C}_{0}\mathcal{D}_{t,[\zeta(t),\omega(t)]}^{\alpha}u(x,t)~=~\frac{\partial^{2} u(x,t)}{\partial x^{2}}-\frac{\partial u(x,t)}{\partial x}+f(x,t),~~(x,t) \in (0,1)\times(0,1), \\
u(x,0)= 0,~~~~ x \in (0,1),\\
u(0,t)= 0,\\ u(1,t)=t^{7}(sin~1) ,~~~ t \in (0,1],
\end{cases}
\end{align}
where $f(x,t)~=~ \frac{\Gamma8}{\Gamma(8-\alpha)}~t^{7-\alpha}sin(x)+t^{7}(sin(x)+cos(x)).$ 	When $\zeta(t)=t$ and $\omega(t)=1$, the exact solution is $u(x,t)=t^{7}sin(x).$  
\end{Example}
To solve this PDE, we use our difference scheme \eqref{sch}. Here, two Tables \ref{8}, \ref{9} are given in support of numerical results. In Table \ref{8}, we take fixed space step size $h=1/2000$ and display maximum absolute errors and errors $E_{2}$ with respect to norm $L_{2}$, also rate of convergence for temporal direction with different time steps $N=8, 16, 32, 64, 128.$~ In Table \ref{9}, we expressed maximum errors and convergence order in space direction to set time steps fix with $\tau=1/500$ and taking different $M=8, 16, 32, 64, 128.$ 
\begin{table}[H]
\centering
\caption{Errors $E_{\infty}$ and $E_{2}$ with convergence rates in time for Example \ref{5.4}, when $h=1/2000$ and different $\alpha$'s.}\label{8}
\begin{tabular}{llllllllllllllllllllllll} \hline
\multicolumn{4}{l}{$\alpha$} & \multicolumn{4}{l}{$N$} & \multicolumn{4}{l}{~~~$E_{\infty}$ error} & \multicolumn{4}{l}{~$R_{t}$} & \multicolumn{4}{l}{~~~$E_{2}$ error} & \multicolumn{4}{l}{$R_{t}$} \\ \hline
0.8     &      &      &      & 8      &     &     &    & 5.2117e-03     &     &     &     &           &     &     &     & 2.2239e-03   &    &    &    &              &    &    &    \\
&      &      &      & 16     &     &     &    & 7.4012e-04     &     &     &     & 2.8159    &     &     &     & 4.1857e-04   &    &    &    & 2.4095   &    &    &    \\
&      &      &      & 32     &     &     &    & 9.1743e-05     &     &     &     & 3.0121    &     &     &     & 5.8367e-05   &    &    &    & 2.8422   &    &    &    \\
&      &      &      & 64     &     &     &    & 1.0658e-05     &     &     &     & 3.1057    &     &     &     & 7.1562e-06   &    &    &    & 3.0279   &    &    &    \\
&      &      &      & 128    &     &     &    & 1.2018e-06     &     &     &     & 3.1486    &     &     &     & 8.2814e-07   &    &    &    & 3.1112   &    &    &    \\
0.5     &      &      &      & 8      &     &     &    & 1.4554e-03     &     &     &     &           &     &     &     & 6.2805e-04   &    &    &    &              &    &    &    \\
&      &      &      & 16     &     &     &    & 1.7116e-04     &     &     &     & 3.0880    &     &     &     & 9.7274e-05   &    &    &    & 2.6907   &    &    &    \\
&      &      &      & 32     &     &     &    & 1.7546e-05     &     &     &     & 3.2862    &     &     &     & 1.1191e-05   &    &    &    & 3.1198   &    &    &    \\
&      &      &      & 64     &     &     &    & 1.6840e-06     &     &     &     & 3.3812    &     &     &     & 1.1323e-06   &    &    &    & 3.3049   &    &    &    \\
&      &      &      & 128    &     &     &    & 1.5961e-07     &     &     &     & 3.3993    &     &     &     & 1.1013e-07   &    &    &    & 3.3620   &    &    &    \\
0.2     &      &      &      & 8      &     &     &    & 2.5522e-04     &     &     &     &           &     &     &     & 1.1012e-04   &    &    &    &              &    &    &    \\
&      &      &      & 16     &     &     &    & 2.5499e-05     &     &     &     & 3.3233    &     &     &     & 1.4486e-05   &    &    &    & 2.9264   &    &    &    \\
&      &      &      & 32     &     &     &    & 2.2221e-06     &     &     &     & 3.5204    &     &     &     & 1.4170e-06   &    &    &    & 3.3538   &    &    &    \\
&      &      &      & 64     &     &     &    & 1.8445e-07     &     &     &     & 3.5906    &     &     &     & 1.2401e-07   &    &    &    & 3.5144   &    &    &    \\
&      &      &      & 128    &     &     &    & 1.8801e-08     &     &     &     & 3.2944    &     &     &     & 1.3014e-08   &    &    &    & 3.2523   &    &    &   \\ \hline
\end{tabular}
\end{table}
\begin{table}[H]
\centering
\caption{The maximum errors and convergence rates $R_{x}$ for Example \ref{5.4}, when  $\tau=1/500$ and differen $\alpha$'s.}\label{9}
\begin{tabular}{cccccccc}
\hline
& \qquad~~$\alpha=0.2$ &            &\qquad~~$\alpha=0.5$ &            &\qquad ~~$\alpha=0.8$ &            \\ \cline{1-7}  
$M $ & ~$E_{\infty}$ error         & $R_{x}$         & ~$E_{\infty}$ error          & $R_{x}$         & ~$E_{\infty}$ error          & $R_{x}$         \\ \hline
8        & 3.0308e-04   &            & 2.7311e-04   &            & 2.3102e-04   &            \\
16       & 7.6029e-05   & 1.9951 & 6.8533e-05   & 1.9946 & 5.8003e-05   & 1.9938 \\
32       & 1.9050e-05   & 1.9968 & 1.7180e-05   & 1.9961 & 1.4560e-05   & 1.9941 \\
64       & 4.7625e-06   & 2.0000 & 4.2962e-06   & 1.9996 & 3.6518e-06   & 1.9953 \\
128      & 1.1909e-06   & 1.9997 & 1.0752e-06   & 1.9985 & 9.2444e-07   & 1.9820 \\ \hline
\end{tabular}
\end{table}	
\begin{table}[H] 
\centering
\caption{The maximum errors and convergence rates $R_{t}$ for Example 5.1 in \cite{cao2015high}, when $h=1/6000$ and spatial convergence rates $R_{x}$, when $\tau=1/200.$ }\label{10} 
\begin{tabular}{cccccccc} \hline
& $\alpha=0.368$ &            & Cao et al. \cite{cao2015high} &            &  $\alpha=0.368$ & &Cao et al. \cite{cao2015high}  \\  \cline{1-4} \cline{5-8}
$N$ & ~$E_{\infty}$ error            & $R_{t}$         & $R_{t}$         & $M$ & ~$E_{\infty}$ error              & $R_{x}$              & $R_{x}$         \\ \hline
8      & 8.1539e-04     &            &            & 4      & 8.8034e-04       &                 &            \\
16     & 7.7720e-05     & 3.3911 & 3.4031     & 8      & 2.2559e-04       & 1.9643      & 1.9643     \\
32     & 6.8845e-06     & 3.4969 & 3.4972     & 16     & 5.6583e-05       & 1.9953      & 1.9953     \\
64     & 5.8739e-07     & 3.5510 & 3.5528     & 32     & 1.4173e-05       & 1.9972      & 1.9972     \\
128    & 4.9904e-08     & 3.5571 & 3.5836     & 64     & 3.5359e-06       & 2.0030      & 2.0030   \\ \hline
\end{tabular}
\end{table}
In Table $10$, we validate the proposed scheme with \cite{cao2015high} (Example 5.1). Firstly, we present the max error, the convergence order with fixed $h=\frac{1}{6000}$ and varying $N$ with $8$, $16$, $32$, $64$, $128$~ for~ $\alpha=0.368$. It is noted that the convergence order of our scheme \eqref{sch} for temporal direction is almost the same as \cite{cao2015high}. After that, we expressed maximum error and convergence rate for spatial dimension to set $\tau=1/200$ and changing $M=4, 8, 16, 32, 64$~ for same value of $\alpha$, we observe that the spatial convergence order of our numerical scheme is same as \cite{cao2015high}. Thus, the scheme presented in \cite{cao2015high} becomes a particular case of the proposed scheme \eqref{sch} for $\zeta(t)=t$ and $\omega(t)=1$. Also, we can compute numerical results to for different suitable choices of scale and weight functions.     
\section{Conclusion}\label{Sec-6}
A high-order numerical scheme based on cubic interpolation formula is discussed for approximation of GCFD of $\alpha$-th order. Properties of discretized coefficients are analyzed and local truncation error in approximation of GCFD is also discussed. Further, we establish a difference scheme for generalized fractional advection-diffusion equation using the developed approximation formula for GCFD. The stability and convergence of this established scheme to approximate the time fractional generalized advection-diffusion equation are studied. Order of accuracy for the difference scheme is $\mathcal{O}(\tau^{4-\alpha}+h^{2}),$ where step-sizes $\tau$ in time and $h$ in space direction. The convergence order of the difference scheme is described by some numerical experiments.\par
Numerical calculations reveal that the proposed difference scheme has $(4-\alpha)-$th order convergence in time; and in space direction, it has second order convergence. The temporal rate of convergence of the scheme for generalized fractional advection-diffusion equation is highest accuracy to date. In addition, the developed scheme can be directly used to get the other Caputo-type of time fractional advection-diffusion equations by selecting the suitable scale $\zeta(t)$ and weight $\omega(t)$ functions in the generalized Caputo-type fractional derivative. The developed scheme is tested for the cases having smooth solutions of the considered fractional advection-diffusion equation. However, the case of the nonsmooth solutions will be presented in our future works.

\section*{Declarations}
\textbf{Funding}: The authors declare that no funds, grants, or other support were received during the preparation of this manuscript.\\\\
\textbf{Conflict of interest}:
The authors declare that they have no known competing financial interests or personal relationships that could have appeared to influence the work reported in this paper.
\pagebreak
\bibliographystyle{unsrt}
\bibliography{Ref1.bibtex}

\end{document}